\documentclass[11pt,letterpaper]{article}

\usepackage[margin=1in]{geometry}

\usepackage{theorem,latexsym,graphicx}
\usepackage{amsmath,amssymb,enumerate}
\usepackage{xspace} 
\usepackage{bm} %
\usepackage{ifpdf}
\usepackage{shadow,shadethm,color}
\usepackage[compact]{titlesec}
\usepackage{algorithm, algorithmic}
\usepackage{subfigure}
\usepackage{verbatim}

\definecolor{Darkblue}{rgb}{0,0,0.4}
\definecolor{Brown}{cmyk}{0,0.81,1.,0.60}
\definecolor{Purple}{cmyk}{0.45,0.86,0,0}
\newcommand{\mydriver}{hypertex}
\ifpdf
 \renewcommand{\mydriver}{pdftex}
\fi
\usepackage[breaklinks,\mydriver]{hyperref}
\hypersetup{colorlinks=true,
            citebordercolor={.6 .6 .6},linkbordercolor={.6 .6 .6},%
citecolor=Darkblue,urlcolor=black,linkcolor=Darkblue,pagecolor=black}
\newcommand{\lref}[2][]{\hyperref[#2]{#1~\ref*{#2}}}


\newtheorem{theorem}{Theorem}[section]
\newtheorem{definition}[theorem]{Definition}

\newtheorem{lemma}[theorem]{Lemma}
\newshadetheorem{lemmashaded}[theorem]{Lemma}

\newtheorem{corollary}[theorem]{Corollary}
\newenvironment{proof}{

\noindent{\bf Proof:}}
{\hfill$\blacksquare$

}

\newcommand{\junk}[1]{}
\newcommand{\ignore}[1]{}

\newcommand{\R}[0]{{\ensuremath{\mathbb{R}}}}
\newcommand{\N}[0]{{\ensuremath{\mathbb{N}}}}

\def\ceil#1{\left\lceil #1 \right\rceil}

\def\set#1{\{ #1 \}}

\newcommand{\pr}[1]{{\rm Pr} \left[ #1 \right]}

\newcommand{\E}{\mathbb{E}}
\newcommand{\Eb}[1]{\E\left[#1\right]} 

\newcounter{note}[section]

\newcommand{\qedsymb}{\hfill{\rule{2mm}{2mm}}}
\renewenvironment{proof}{\begin{trivlist} \item[\hspace{\labelsep}{\bf
\noindent Proof.\/}] }{\qedsymb\end{trivlist}}%

\newcommand{\initOneLiners}{%
    \setlength{\itemsep}{0pt}
    \setlength{\parsep }{0pt}
    \setlength{\topsep }{0pt}
}

\newcommand{\ve}[1]{{\cal V}\left(#1\right)} 
\newcommand{\extends}{\succeq} 

\newcommand{\piecewise}[1]{\begin{cases}#1\end{cases}}
\newcommand{\tif}{\text{if }}

\begin{document}

\title{Concentration and Moment Inequalities for Polynomials of Independent Random Variables\thanks{An extended 10-page abstract of this paper appeared in Proceedings of the Symposium on Discrete Algorithms (SODA2012)}}

\author{Warren Schudy \thanks{IBM T. J. Watson Research Center, P.O.
Box 218, Yorktown Heights, NY
     10598. {\tt wjschudy@us.ibm.com}} \and Maxim Sviridenko\thanks{University of Warwick, {\tt sviri@dcs.warwick.ac.uk}} }
\date{}

\maketitle

\begin{abstract}
In this work we design a general method for proving moment inequalities for polynomials of independent random variables. Our method works for a wide range of random variables including Gaussian, Boolean, exponential, Poisson and many others. We apply our method to derive general concentration inequalities for polynomials of independent random variables. We show that our method implies concentration inequalities for some previously open problems, e.g.\ permanent of random symmetric matrices. We show that our concentration inequality is stronger than the well-known concentration inequality due to Kim and Vu \cite{KV}. The main advantage of our method in comparison with the existing ones is a wide range of random variables we can handle and bounds for previously intractable regimes of high degree polynomials and small expectations.
On the negative side we show that even for boolean random variables each term in our concentration inequality is tight.
\end{abstract}

\section{Introduction}

Concentration and moment inequalities are vital for many applications in Discrete Mathematics, Theoretical Computer Science, Operations Research, Machine Learning and other fields. In the classical setting we have $n$ independent random variables $X_1,\dots, X_n$ and we are interested in a behavior of a function $f(X_1,\dots, X_n)$ of these random variables. Probably, the first concentration inequality with exponential bounds for tails was proven by S.\ Bernstein \cite{B1}
 who showed that if $X_i$ are random variables that take values $+1$ or $-1$ with probability $1/2$ (i.e. Rademacher random variables) then
 $$Pr\left[ \left|\sum_{i=1}^n X_i\right|\ge \varepsilon n\right]\le 2e^{-\frac{\varepsilon^2 n}{2(1+\varepsilon/3)}}.$$
 More general inequalities known as Chernoff Bounds became part of the mathematical jargon to the extent that many papers in Theoretical Computer Science use them without stating the inequalities. In the last 20 years this area of Probability Theory and related area of mathematics  studying the measure concentration  has flourished driven by the variety of applications and settings. The surveys and books \cite{survey1,survey2,survey3,mcdiarmid98,Ledoux} provide the historical and mathematical background in this area.

The most general and powerful methods known up to date to prove such inequalities is Ledoux's entropy method \cite{Ledoux} and the famous Talagrand's isoperimetric inequality \cite{T1}. Yet as was noticed by Vu \cite{V1} these methods and corresponding inequalities  work well only when the Lipschitz coefficients of the function $f(X_1,\dots, X_n)$ are relatively small.  The standard example showing the weakness of such methods is the number of triangles in random graphs $G(n,p)$. Until the concentration inequality due to Kim and Vu \cite{KV} no non-trivial concentration of this function about its mean was known.

Kim and Vu \cite{KV} introduced the notion of average Lipschitz coefficients  based on the partial derivatives of a polynomial evaluated at the point $(\E[X_1],\dots, \E[X_n])$ (in the multilinear case). These new parameters enabled them to prove a concentration inequality for polynomials of boolean random variables. This inequality has been applied to the problem of approximately counting triangles in (e.g.) a social network by sampling the edges \cite{DOULION,DOULION2}, to average-case correlation clustering \cite{CCluster}, and to a variety of other applications \cite{V1}. The original inequality from \cite{KV} was tightened and generalized in \cite {V1} to handle arbitrary random variables in the interval $[0,1]$. Yet the inequality from \cite{V1} did not work well for high degree polynomials and for random variables $f(X_1,\dots, X_n)$ with small expectation. The follow up work by Vu \cite{V2} handles the case of polynomials with small expectation and extremely small smoothness parameters.

On the other side the concentration of polynomials of Gaussian and Rademacher random variables has long been a subject of interest in Probability Theory. The moment and concentration inequalities for polynomials of centered Gaussians are known as Hypercontractivity Inequalities \cite{hist,J}. We discuss various inequalities known in this setting and their connection to our results in Section \ref{connection}. Recently, the Hypercontractivity Inequalities and their ``anti-concentration'' counterparts found many applications in Theoretical Computer Science and Machine Learning \cite{ap1,ap2,ap3,ap4,ap5,ap6}.

The above motivated us to study the moment and concentration inequalities for polynomials of independent random variables. We design a general method that works for a wide range of random variables including Gaussian, Boolean, exponential, Poisson and many others (see Section \ref{sec:exampleMom} for more examples). We show that our method implies concentration inequalities for some previously open problems, e.g.\ permanent of random symmetric matrices. We also show that our main concentration inequality is stronger than the well-known concentration inequality due to Kim and Vu \cite{KV}. On the negative side we show that even for boolean random variables each term in our concentration inequality is tight.

\subsection{Our Results}
For a cleaner exposition we first describe our results in the restricted setting of multilinear polynomials with non-negative coefficients.
We are given a hypergraph $H=({\cal V}(H),{\cal H}(H))$ consisting of a set ${\cal V}(H)=\{1,2,\ldots,n\}=[n]$ of vertices and a set ${\cal H}(H)$ of hyperedges. A hyperedge $h$ is a set $h \subseteq {\cal V}(H)$ of $|h| \le q$ vertices. We are also given a non-negative weight $w_h$ for each $h \in {\cal H}(H)$. For each such weighted hypergraph and real-valued weight $w_h$ for its hyperedges, we define a polynomial
\begin{align}
f(x)&= \sum_{h\in {\cal  H}(H) } w_h\prod_{v\in h}x_v. \label{eqn:fmulti}
\end{align}

Our smoothness parameters were strongly motivated by  the average partial derivatives introduced by Kim and Vu \cite{KV,V1}.
For any $y\in \R^n$, hypergraph $H$, nonnegative weights   $ w$, and $h_0 \subseteq {\cal V}(H)$ let
$$\mu(y,H,w,h_0)=\sum_{h\in {\cal  H}(H) ~|~ h \supseteq h_0} w_h \prod_{v\in h\setminus h_0}|y_v|.$$
Note that $h_0$ need not be a hyperedge of $H$ and may even be the empty set.
Also note that $\mu(y,H,w,h_0)$ is equal to the $|h_0|$-th partial derivative of polynomial $f(x)$ with respect to each variable $x_v$ for $v \in h_0$, evaluated at the point $x=y$ if $y\in R^n_+$.
For a given collection of independent random variables $Y=(Y_1,\dots,Y_n)$, hypergraph $H$, integer $r\ge 0$ and nonnegative weights   $ w$, we define
$$\mu_r=\mu_r(H,w)=\max_{h_0\subseteq [n]:|h_0|=r}\E\left[\mu(Y,H,w, h_0)\right] =  \max_{h_0\subseteq [n]:|h_0|=r}\left(\sum_{h\in {\cal  H}(H)| h \supseteq h_0} w_h \prod_{v\in h\setminus h_0}\Eb{|Y_v|}\right),$$
where we used the independence of random variables $Y_v$ in the last equality. Sometimes we will also use the notation $\mu_r(f)=\mu_r(H,w)$. Note that when the $Y_v$ are non-negative $\mu_r$ is equal to the maximal expected partial derivative of order $r$ of the polynomial $f(x)$, which was the parameter used in the Kim-Vu concentration inequalities \cite{KV,V1}.

Our concentration inequalities will hold for a general class of independent random variables including most classical ones.
\begin{definition}\label{var}
A random variable $Z$ is called \emph{moment bounded with parameter} $L>0$ if for any integer $i\ge 1$,
\[
\Eb{|Z|^i} \le i \cdot L \cdot \Eb{|Z|^{i-1}}
.\]
\end{definition}
Roughly speaking a random variable $Z$ is moment bounded with parameter $L$ if $\Eb{|Z|} \le L$ and the tails of its distribution decay no slower than an exponentially distributed random variable's tails do. Indeed note that Definition \ref{var} implies that any moment bounded random variable $Z$ satisfies $\Eb{|Z|^i} \le L^ii!$.
In Section \ref{sec:exampleMom} we show that three large classes of random variables are moment bounded: bounded, continuous log-concave \cite{BB,An} and discrete log-concave \cite{An}. For example the Poisson, binomial, geometric, normal (i.e.\ Gaussian), and exponential distributions are all moment bounded.

We prove the following:
\begin{theorem}\label{main1special}
We are given $n$ independent moment bounded random variables $Y=(Y_1,\dots, Y_n)$ with the same parameter $L$. We are given a multilinear polynomial $f(x)$ with nonnegative coefficients of total power\footnote{We reserve the more traditional terminology of ``degree'' for the number of neighbors of a vertex in a hypergraph.} $q$.  Let $f(Y)=f(Y_1,\dots,Y_n)$  then
$$Pr\left[|f(Y)-\E[f(Y)]|\ge \lambda\right]\le e^2\cdot \max\left\{\max_{r=1,\dots,q} e^{-\frac{\lambda^2}{\mu_0\mu_r\cdot L^r\cdot R^q }},\max_{r=1,\dots,q}e^{-\left(\frac{\lambda}{\mu_r\cdot L^r \cdot R^q}\right)^{1/r}}\right\},$$
where $R\ge 1$ is some absolute constant.
\end{theorem}

We also show that Theorem \ref{main1special} is the best possible bound as a function of these parameters, up to logarithms in the exponent and dependence of the constants on the total power $q$. This lower bound holds even for the well-studied special case where the random variables take the values 0 and 1 only, which we show in Section \ref{sec:exampleMom} to be moment bounded with parameter 1.
\begin{theorem} \label{thm:LB}
For any $q \in \N$, real numbers $\mu_0^*,\mu_1^*,\ldots,\mu_q^*>0$ and $\lambda > 0$ there exist independent 0/1 random variables $X=X_1,\dots,X_n$ and a polynomial $f(x)$ of power $q$ such that $\mu_i(f) \le \mu_i^*$ for all $0 \le i \le q$ and
\begin{align}\label{eq:LB}
Pr\left[f(X)\ge \Eb{f(X)} + \lambda\right] &\ge \max_{r=1,\dots,q} \max\left\{e^{-\left(\frac{\lambda^2}{\mu^*_0\mu^*_r} + 1\right)\log C},e^{-\left(\left(\frac{\lambda}{\mu^*_r }\right)^{1/r} +1\right)\log C}\right\}
\end{align}
where $C=c_0\Lambda_1^{c_1}\Lambda_2^{c_2}\Lambda_3^{c_3}$, $c_0$, $c_1$, $c_2$ and $c_3$ are absolute constants, $\Lambda_1 =\max_{0 \le i,j \le q} (\mu_i^*/\mu_j^*)^q$, $\Lambda_2=\max_{1 \le i \le q} \lambda/\mu_i^*$, and $\Lambda_3=q^q$.
\end{theorem}

We generalize Theorem \ref{main1special} in two ways. Firstly, we allow negative coefficients. Secondly, we remove the restriction for a polynomial to be multilinear, instead allowing each monomial to have total power at most $q$ and maximal power of each variable at most $\Gamma$. For example $X_1^2 X_2^4 X_3^1$ has total power $q=7$ and maximal variable power $\Gamma=4$ and the multilinear case is when maximal power is $\Gamma=1$.
We defer the formal definition of general polynomials and the appropriate generalization of $\mu_r$ to Section \ref{sec:defs}.

Our main result in this paper is the following:
\begin{theorem}\label{main1}
We are given $n$ independent moment bounded random variables $Y=(Y_1,\dots, Y_n)$ with the same parameter $L$. We are given a general polynomial $f(x)$ of total power $q$ and maximal variable power $\Gamma$.  Let $f(Y)=f(Y_1,\dots,Y_n)$  then
$$Pr\left[|f(Y)-\E[f(Y)]|\ge \lambda\right]\le e^2\cdot \max\left\{\max_{r=1,\dots,q} e^{-\frac{\lambda^2}{\mu_0\mu_r \cdot L^r  \cdot \Gamma^r\cdot R^q}},\max_{r=1,\dots,q}e^{-\left(\frac{\lambda}{\mu_r \cdot L^r\cdot\Gamma^r\cdot R^q}\right)^{1/r}}\right\},$$
where $R\ge 1$ is some absolute constant.
\end{theorem}

For large power polynomials the concentration bounds in the Theorem \ref{main1} may not provide interesting concentration bounds due to the term $R^q$  in the exponent, yet we believe that the moment computation method developed in this paper is useful even in this setting. We show two specific examples when our method works. Our first example is a concentration inequality for permanents of random matrices. The anti-concentration counterpart was recently studied by Aaronson and Arkhipov \cite{AA} in the Gaussian setting and by Tao and Vu  \cite{TV} in the setting with Rademacher random variables.
\begin{theorem}\label{perm}
We are given $n\times n$ matrix $A$ with random entries $Y_{ij}$ which are independent moment bounded random variables with parameter $L=1$ and $\E[Y_{ij}]=0$. Let $P(A)$ be the permanent of the matrix $A$  then
$$Pr[|P(A)|\ge t \sqrt{n!}]\le  \max\left\{e^{-n},e^2\cdot e^{-c\cdot t^{2/n}}\right\}$$
for some absolute constant $c>0$ and parameter $t>0$.
\end{theorem}
 Our next example is an analogous Theorem for the permanent of a random symmetric matrix.
\begin{theorem}\label{perm1}
We are given $n\times n$ symmetric matrix $A$ with random entries $Y_{ij}$ which are independent moment bounded random variables for all pairs $(i,j)$ with $i\le j$ with parameter $L=1$ and $\E[Y_{ij}]=0$. Let $P(A)$ be the permanent of the matrix $A$  then
$$Pr[|P(A)|\ge t \sqrt{n!}]\le  \max\left\{e^{-n},e^2\cdot e^{-c\cdot t^{2/n}}\right\}$$
for some absolute constant $c>0$ and parameter $t>0$.
\end{theorem}
Note that the above concentration inequalities can be easily derived from the Hypercontractivity Inequality in the special case of Gaussian and Rademacher random variables (Theorem \ref{hyper}).

\subsection{Applications in Randomized Rounding for Mathematical Programming Problems}

As we noted all current methods to prove concentration bounds for polynomials do not work well for high power polynomials. Another feature that makes current concentration methods fail is low expectation.
One application where such concentration bounds could be applied is in design and analysis of randomized rounding algorithms for non-linear mathematical programming problems. 

Many real-life optimization problems can be formulated using integer programming which is well-known to be computationally intractable (NP-hard). One way to solve such a problem both in theory and practice is to consider a linear programming relaxation, solve it using one of the standard methods and use the fractional optimal solution as a guidance in finding an integral solution of good quality.  The seminal paper of Raghavan and Thompson \cite{RT} suggested to round each boolean variable to one with probability $x^*_i$ and to zero with probability $1-x^*_i$ independently at random where $x^*$ is the optimal fractional solution. The analysis of such algorithms is based on applying Chernoff Bounds to each constraint of the integer program separately and then applying a union bound over all the constraints. Such a method proved to be useful for a wide range of models and led to approximation algorithms that still have best known performance guarantees today.

A natural generalization of this framework is to apply it to non-linear optimization models. Many such problems are still computationally tractable if we replace the constraint that variables must be boolean $x_i\in \{0,1\}$ with continuous constraints  $0\le x^*_i\le 1$, e.g. quadratic convex constraints. There are many real-life optimization problems with constraints and objective functions modeled  in such a way, e.g. we would like to optimize a congestion for a group of edges in a multi-commodity flow problem in a "fair" way, i.e. we don't want to have one edge to get significantly higher congestion than the other. The standard way to ensure that in practice is to optimize (or constrain) sum of   squared congestions over the edges in a that group. The constraints generated this way are convex quadratic constraints and continuous optimization problems with such constraints are polynomially solvable.

To analyze the randomized rounding framework for such mathematical programming models one  needs to apply concentration inequality to each non-linear constraint. If the size of the group of edges for which we are trying to optimize the total congestion in a fair way is sub-logarithmic and each edge in the fractional solution has a constant congestion (a situation quite natural from application viewpoint) then our concentration inequalities would be the only available tool to analyze such an algorithm.

\subsection{Sketch of Our Methods}\label{sec:methods}

Most concentration results for non-negative random variables are proven using Markov's inequality as follows:
\begin{align}
\pr{Z \ge \lambda} &= \pr{g(Z) \ge g(\lambda)} \le \frac{\Eb{g(Z)}}{g(\lambda)} \label{eqn:genericMarkov}
\end{align}
where $Z$ is the random variable that we are trying to show concentration of and $g$ is either $g(z)=z^k$ for some positive even integer $k$, $g(z)=e^{tz}$ for some real $t>0$, or some other non-negative increasing function $g$. One then computes an upper bound on either the $k$th moment $\Eb{g(Z)}=\Eb{Z^k}$ or the moment generating function $\Eb{g(Z)}=\Eb{e^{tZ}}$. Chernoff bounds are proven using moment generating functions, so it would be most natural to use moment generating functions to prove our bounds as well. Unfortunately the tails of the distribution of polynomials can be sufficiently large to make the moment generating function $\Eb{e^{tZ}}$ infinite for all $t>0$. Kim and Vu worked around this issue by applying (\ref{eqn:genericMarkov}) not to the polynomial itself but to various auxilliary random variables with better behaved tails. Unfortunately a union bound over these auxiliary variables introduced an extraneous factor logarithmic in the number of variables into their bounds (see Section \ref{connection} for a comparison of our results to theirs). We avoid this issue by computing moments instead of the moment generating function.

We now give an instructive half-page bound on the second moment of a multilinear polynomial $f(X) = \sum_{h\in {\cal  H}} w_h\prod_{v\in h}X_v$ where all $\Eb{X_v}=0$, $X_v$ are moment bounded with parameter $L$, and all $h \in {\cal H}$ have $|h|=q$ and $w_h \ge 0$. Using definitions, linearity of expectation, and independence we get
\begin{align}
\Eb{f(X)^2} & = \Eb{\sum_{h_1\in {\cal  H}} \sum_{h_2\in {\cal  H}} w_{h_1} w_{h_2} \left(\prod_{v \in h_1} Y_v \right) \left(\prod_{v \in h_2} Y_v \right)} \nonumber\\
& = \sum_{h_1\in {\cal  H}} \sum_{h_2\in {\cal  H}} w_{h_1} w_{h_2}\prod_{v \in (h_1 \cup h_2)} \Eb{Y_v^{d_v}} \label{eqn:begin}
\end{align}
where $d_v \in \{1,2\}$ is the number of $h_i \ni v$. Now if $d_v=1$ for any $v$ we have $\Eb{Y_v^{d_v}}=\Eb{Y_v}=0$, so the only non-zero terms of the sum (\ref{eqn:begin}) are when $h_1=h_2$. We therefore get
\begin{align}
\Eb{f(X)^2} & = \sum_{h\in {\cal  H}} w_{h} w_{h}\prod_{v \in h} \Eb{Y_v^2} \nonumber\\
& \le \sum_{h\in {\cal  H}} \mu_q w_{h} \prod_{v \in h} (2L\Eb{|Y_v|}) \nonumber\\
& = (2L)^q \mu_q \sum_{h\in {\cal  H}} w_{h} \prod_{v \in h} \Eb{|Y_v|} \nonumber\\
& = (2L)^q \mu_q \mu_0 \label{eqn:variance}
.\end{align}
where we used the fact that $\Eb{Y_v^2} \le 2 L \Eb{|Y_v|}$ from moment boundedness and $w_h \le \max_h w_h = \mu_q$ from the definition of $\mu_q$. Combining (\ref{eqn:variance}) with Markov's inequality (\ref{eqn:genericMarkov}) yields
\[
\pr{|f(X)| \ge \lambda} \le \left(\frac{\lambda^2}{(2L)^q \mu_q \mu_0}\right)^{-1}
.\]
which is comparable to the $e^{-\lambda^2/(\mu_0 \mu_q (RL)^q)}$ term in Theorem \ref{main1special} for small $\lambda$. In order to get exponentially better bounds for larger $\lambda$ we will compute higher moments.

Now we outline what we do differently to handle higher moments and general polynomials.

The first step is to express polynomial $f$ over variables $Y_v$ as a sum of polynomials $g^{(1)},\dots,g^{(m)}$ over variables $Y_v^\tau - \Eb{Y_v^\tau}$ for various $1 \le \tau \le q$. The main task is bounding the moments of each of these polynomials. We later combine these bounds to get a bound on the moment of $f$. Each of the centered polynomials has $\Eb{Y_v^\tau - \Eb{Y_v^\tau}}=0$, which takes the place of the $\Eb{Y_v}=0$ in the above special case. We also ensure that each $g^{(i)}$ has non-negative weights.

Bounding moments of some $g_i$ begins by expanding $\Eb{g_i^k}$ similar to (\ref{eqn:begin}) with a sum over $h_1,\dots,h_k$. As before only terms of the sum where every vertex $v$ occurs in $d_v \ge 2$ different hyperedges are non-zero, but this is no longer equivalent to the simple condition $h_1=h_2$.

We find it helpful to separate the structure of the hyperedges $h_1,\dots,h_k$ from the identity of the variables involved. We therefore generate $h_1,\dots,h_k$ by composing two processes: first generate $h_1,\dots,h_k$ over vertex set $[\ell]$ for every $\ell \ge 1$ and then consider every possible embedding of those artificial vertices into the vertex set $[n]$. For a fixed sequence of hyperedges over vertex set $[\ell]$ we do arguments analogous to (\ref{eqn:variance}) to get a product of various $\mu_i$ and $L$. This bound is a function of the number of connected components $c$ in $h_1,\dots,h_k$. Finally we do some combinatorics to prove a \emph{counting lemma} on the number of possible $h_1,\dots,h_k$ with vertex set $[\ell]$ with all degrees at least two and $c$ connected components.

One additional complication is that we need to use moment boundedness to bound moments of order much larger than the second moments $\Eb{Y_v^2} \le 2 L \Eb{|Y_v|}$ we used in the above special case. If we treated the factor that replaces that $2$ as a constant that would make the constant $R$ in our final bounds linear in $q$ instead of an absolute constant. Fortunately these extra factors are small for \emph{most} of the possible $h_1,\dots,h_k$, which enables our counting lemma to absorb these extra factors.

Our lower bounds are based on lower-bounding the concentration of certain concrete polynomials. It is well known that Chernoff bounds are essentially tight, i.e.\ a sum of $n$ i.i.d.\ 0/$M$ random variables each with expected value $\mu/n$ has probability roughly $e^{-\lambda^2/(2\mu M)}$ of exceeding its mean by $\lambda \le \mu$. Our lower bound of $e^{-\tilde O(\lambda^2/(\mu_0 \mu_r))}$ follows from a degree $q$ polynomial that acts like this linear polynomial with $M=\mu_r$ and $\mu=\mu_0$.
The idea behind the lower bound corresponding to $e^{-(\lambda/\mu_r)^{1/r}}$ is the fact that $\pr{(\sum_i X_i)^r \ge \lambda} = \pr{\sum_i X_i \ge \lambda^{1/r}} = e^{-\tilde \Theta(\lambda^{1/r})}$ where $\sum_i X_i$ is binomially distributed with mean 1. Our lower bound does similar arguments with a multilinearized version of $(\sum_i X_i)^r$.

\subsection{Definitions}\label{sec:defs}

We now state the generalizations of the notations given in the introduction for general polynomials.

A \emph{powered hypergraph} $H$ consists of a set ${\cal V}(H)$ of vertices and a set ${\cal H}(H)$ of powered hyperedges.
A \emph{powered hyperedge} $h$ consists of a set $\ve{h} \subseteq {\cal V}(H)$ of $|\ve{h}|=\eta(h)$ vertices and an $\eta(h)$-element \emph{power vector} $\tau(h)$ with one strictly positive integer component $\tau(h)_v=\tau_{hv}$ per vertex $v \in \ve{h}$. We will hereafter omit the ``powered'' from ``powered hypergraph'' and ``powered hyperedge'' since we have no need to refer to the basic hypergraphs used in the introduction.
For any powered hyperedge $h$ we let $q(h)=\sum_{v \in \ve{h}}\tau_{hv}$.
For each such powered hypergraph $H$ and real-valued weights $w_h$ for its hyperedges, we define a polynomial
\begin{align}
f(x) &= \sum_{h\in {\cal  H}(H)} w_h\prod_{v\in {\cal V}(h)}x_v^{\tau_{hv}}\label{eqn:f}
.\end{align}
The hyperedge $h$ corresponds to a monomial $\prod_{v\in h}x_v^{\tau_{hv}}$.
The parameters $q(h)$ and $\eta(h)$ will be called the \emph{total power} and \emph{cardinality} of the hyperedge $h$ (or monomial corresponding to $h$). Let $\Gamma=\max_{h\in {\cal  H}(H), v\in h}\tau_{hv}$ be the maximal power of a variable in polynomial $f(x)$, e.g.\ $\Gamma=1$ for multilinear polynomials.
We assume, by convention, that $\prod_{i\in \emptyset}x_i=1$. Since the variables in our polynomials are indexed by vertices in our hypergraphs we use the terms ``variable'' and ``vertex'' interchangeably.

For powered hyperedges $h_1$ and $h_2$ (not necessarily hyperedges of a hypergraph) we write $h_1 \extends h_2$ if ${\cal V}(h_1) \supseteq {\cal V}(h_2)$ and $\tau_{h_1v} = \tau_{h_2v}$ for all $v \in {\cal V}(h_2)$. In the context of hypergraph $H$ with vertex set $[n]$ clear from context, for a given collection of independent random variables $Y=(Y_1,\dots,Y_n)$, integer $r\ge 0$  and weights $w$ we define
\begin{eqnarray}\label{def:mu}
\mu_r(w,Y)=\max_{h_0|\ {\cal V}(h_0)\subseteq {\cal V}(H),\ q(h_0)=r}\left(\sum_{h\in {\cal  H}(H)| h \extends h_0} |w_h|\prod_{v\in \ve{h}\setminus \ve{h_0}}\Eb{|Y_v^{\tau_{hv}}|}\right)
\end{eqnarray}
where $h_0$ ranges over all possible powered hyperedges with vertices from $[n]$ with total power $q(h_0)=r$. The cardinality of $h_0$ is not explicitly restricted but it cannot exceed $r$ since the powers $\tau_{h_0v}$ are strictly positive integers summing to $q(h_0)$.
We will sometimes write $\mu_r(f,Y)$ for polynomial $f(Y)$ instead of $\mu_r(w,Y)$ to emphasize the dependence on polynomial $f(Y)$.
If we write $\mu_r(f)$ for a polynomial $f$ this means $\mu_r(w,Y)$ for the weight function $w$ and random variable vector $Y$ corresponding to $f$ as in (\ref{eqn:f}). If the polynomial is clear from context we write simply $\mu_r$.
In the special case that all coefficients are non-negative $\mu_r$ is upper bounded by the maximal expected partial derivative of order $r$ of the polynomial $f(x)$, and this bound is loose for two reasons. First we do not have multipliers that depend on powers that are present in derivatives. Second we throw away some positive terms that are present in derivatives since we enforce $\tau_{hv} = \tau_{h_0 v}$ for all $v \in {\cal V}(h_0)$, whereas derivatives would consider all $h$ with $\tau_{hv} \ge \tau_{h_0 v}$. For example for the polynomial $Y_0^3 Y_1^3 + Y_1^2$ of non-negative random variables we have $\mu_2=1$ while the maximal expected second partial derivative is equal to
$$\max\left\{2+6\Eb{Y_0^3}\Eb{Y_1},6\Eb{Y_0}\Eb{Y_1^3},9\Eb{Y_0^2}\Eb{Y_1^2}\right\}.$$ Overall, our definition of smoothness is a bit tighter (although less natural) than the partial derivatives used in \cite{KV}, \cite{V1} that inspired it. We decided to use it since it naturally arises in our analysis.

\subsection{Comparison with Known Concentration Inequalities}\label{connection}
There are many concentration inequalities dealing with the case when we are interested in a sum of weakly dependent random variables. The paper \cite{JR} provides a good survey and comparison of various inequalities for that setting. Below we will survey only known concentration inequalities for the case of polynomials of independent random variables. The previous works were dealing either with the case of boolean random variables, variables distributed in the interval $[0,1]$, Gaussian random variables or log-concave random variables.
\subsubsection{Comparing with the Kim-Vu inequality}
Probably the most famous concentration inequality for polynomials is due to Kim and Vu \cite{KV} published in 2000. There are many variants, extensions and equivalent formulations of that inequality. We consider a variant from the survey paper by Vu \cite{V1} (Theorem 4.2 in Section 4.2).
\begin{theorem}[Kim-Vu Concentration Inequality]\label{KimVu}
Consider a polynomial $f(Y)=f(Y_1,\dots,Y_n)$ with coefficients in the interval $[0,1]$. We denote $\partial_A f(Y)$ a polynomial obtained from $f(Y)$ by taking partial derivatives with respect to $A$ where $A$ is a multiset of indices probably with repetitions. Let $Y_1,\dots,Y_n$ be independent random variables with arbitrary distributions on the interval $[0,1]$. Let $q$ be the degree of polynomial $f(Y)$ and
$\E_j[f(Y)]=\max_{|A|\ge j}\E[\partial_A f(Y)]$. Assume we are given an integer $q'\le q$ and a collection of positive numbers
${\cal E}_0\ge  {\cal E}_1\ge \dots \ge {\cal E}_{q'}=1$ and $\lambda$ satisfying
\begin{enumerate}
\item ${\cal E}_j\ge \E_j[f(Y)]$ for $j=0,\dots,q'$;
\item ${\cal E}_j/{\cal E}_{j+1}\ge \lambda +4j\log n$ for $j=0,\dots, q'-1$;
\end{enumerate}
then the following holds
$$Pr[|f(Y)-\E[f(Y)]|\ge c_q\sqrt{\lambda {\cal E}_0{\cal E}_1}] \le d_q e^{-\lambda/4}$$
where $c_q\approx q^{q/2}$ and $d_q=2^{q+1}-2$ (see precise definitions in \cite{V1}).
\end{theorem}
This stronger version of the original Kim-Vu inequality \cite{KV} has dependence on parameter $q'$ which could be helpful for some applications (see discussion in \cite{V1}). We compare below our inequality with the inequality in Theorem \ref{KimVu} when $q'=q$ which includes the original  Kim-Vu inequality \cite{KV} and is the most relevant variant in terms of various applications (our inequality does not seem to be comparable with the general version of the Theorem \ref{KimVu}).

Re-writing our inequality from Theorem \ref{main1} in the same form we could derive
$$Pr\left[|f(Y)-\E[f(Y)]|\ge \max_{r=1,\dots,q} \max\left\{ \sqrt{\tau \mu_0\mu_r \cdot L^r  \cdot \Gamma^r\cdot R^q}, \tau^r \mu_r \cdot L^r \cdot \Gamma^r\cdot R^q \right\} \right]\le e^2\cdot e^{-\tau}$$
instead of the bound in the Theorem \ref{main1} for any $\tau>0$. Using the properties of bounds ${\cal E}_j$ we derive
${\cal E}_0\ge  \lambda^{j}\E_j[f(Y)]$ and ${\cal E}_1\ge  \lambda^{j-1}\E_j[f(Y)]$. In addition, as we already noticed, our definition of smoothness is tighter than the one based on partial derivatives, i.e. ${\cal E}_j\ge \E_j[f(Y)]\ge \mu_j$. Therefore,
 $$\sqrt{\lambda {\cal E}_0{\cal E}_1}\ge \max_{r=1,\dots,q} \max\left\{ \sqrt{\lambda \mu_0\mu_r }, \lambda^r \mu_r\right\}.$$
Choosing $\tau=\lambda/4$, we obtain that the concentration inequality of Theorem \ref{main1} implies the inequality from Theorem \ref{KimVu} (we don't explicitly specify the relationship between our  absolute constant $R$ and constants used in the definition of $c_q$ and $d_q$).  We list below the various ways our inequality generalizes or tightens the inequality from Theorem \ref{KimVu}.
\begin{enumerate}
\item The bounds in our inequality do not depend on the total number of random variables $n$ while all variants of the Kim-Vu inequality have this dependence due to the usage of the union bound in their proof.
\item Our inequality covers a much wider range of random variables, including most commonly used ones not just the variables distributed in the interval $[0,1]$.
\item Our definition of smoothness while being related to (and strongly motivated by) the smoothness based on partial derivatives is tighter and for some applications involving polynomials with large $\Gamma$ will provide a better concentration bound.
\item Our bounds have a better dependence on the degree of the polynomials. We also introduce a parameter $\Gamma$ that is a maximal power of a variable in a polynomial which leads to substantially tighter bounds for the most important special case of multilinear polynomials.
\end{enumerate}

Another concentration inequality that appeared in the literature is due to Boucheron et al.\ \cite{BBLM} (Section 10).
\begin{theorem}\label{BBLM}
Consider a multilinear degree $q$ polynomial $f(Y)=f(Y_1,\dots,Y_n)$  of the independent boolean random variables  $Y_1,\dots,Y_n$.  Then
$$Pr[f(Y)\ge \E[f(Y)]+\lambda] \le e^{-\frac{1}{R\cdot q}}\max\left\{  \max_{r=1,\dots,q}e^{-\left(\frac{\lambda^2}{16q^2\mu_0\mu_r}\right)^{1/r}}, \max_{r=1,\dots,q}e^{-\left(\frac{\lambda}{4q \mu_r}\right)^{1/r}}\right\}$$
for some absolute constant $R>0$.
\end{theorem}
The second term in the maximum looks very similar to ours in the Theorem \ref{main1} but the first term is substantially higher due to the power $1/r$. Also their inequality does not seem to generalize to general class of random variables considered in Theorem \ref{main1}. Note that the Theorem \ref{BBLM} is just a corollary of a moment inequality proved for  much more general functions than polynomials.

\subsubsection{Gaussian and Rademacher Random Variables}
Another class of known concentration inequalities deals with the case when random variables are either centered (or zero mean) Gaussians or variables that have value $+1$ or $-1$ with probability $1/2$ (such random variables are often called Rademacher random variables). The history of moment and concentration inequalities in this setting is quite rich, we refer the reader to the Lecture 16 in Ryan O'Donnell Lecture Notes on Boolean Analysis \cite{hist} or the book by S. Janson \cite{J} (Sections V and VI). We will call the moment and corresponding concentration inequalities the Hypercontractivity Inequalities for the formal proofs see Theorems 6.7 and 6.12 in \cite{J}.
\begin{theorem}[Hypercontractivity Concentration Inequality]\label{hyper}
Consider a degree $q$ polynomial $f(Y)=f(Y_1,\dots,Y_n)$ of independent centered Gaussian or Rademacher random variables $Y_1,\dots,Y_n$.  Then
$$Pr[|f(Y)-\E[f(Y)]|\ge \lambda] \le e^2\cdot e^{-\left(\frac{\lambda^2}{R\cdot Var[f(Y)]}\right)^{1/q}},$$
where $Var[f(Y)]$ is the variance of the random variable $f(Y)$ and $R>0$ is   an absolute constant.
\end{theorem}

It is well-known that functions of Gaussian random variables are better concentrated around their mean than for example functions of Boolean random variables even in such a simple case as a sum of independent random variables. Therefore, in general we cannot expect to match the bound of Theorem \ref{hyper} in the setting of moment bounded random variables.
Nevertheless, if $M_2\approx \max_{r\in [q]}\mu_0\mu_r$ (e.g.\ it happens when power $q=O(1)$, the polynomial is multilinear, $\mu_r = O(1)$ for $r\in [q-1]$ and $w_h\in \{0,1\}$ for all hyperedges $h$) and $\lambda\le \mu_0$ then Theorem \ref{main1} provides a better concentration bound even in this setting.

  An interesting concentration inequality for degree $q$ polynomials of centered Gaussian random variables was recently proven by R. Latala \cite{L}. This inequality generalizes the previously known inequalities for the case when $q=2$ \cite{HW}. The papers by Major \cite{major} and Lehec \cite{lehec} simplify and explain Latala's proof. Latala uses certain smoothness parameters that seem to be natural only in the setting of continuous  random variables. We do not see the way to define similar smoothness parameters in the setting of general moment bounded (or even boolean) random variables.

\subsubsection{Log-Concave  Random Variables}

We define log-concave random variables in the Section \ref{sec:exampleMom} and give many examples of such variables. Latala and Lochowski \cite{LL} consider the setting with non-negative log-concave random variables and multi-linear polynomials. Recently, Adamczak and  Latala  \cite{AL} considered symmetric log-concave random variables and polynomials of degree at most three and  symmetric exponential random variables (or variables having Laplace distribution) for polynomials of arbitrary degree.  The main drawback of their approach in \cite{LL} is that they estimate tails of  random variables instead of estimating the deviation from the mean which is required in most applications. We can show that their smoothness parameters can be derived from ours $\mu_r$ (and the tail bounds) in the case of exponential random variables or any random variables that are tight for our moment boundness condition, i.e.\ $\Eb{|X|^{i}}\approx i L \Eb{|X|^{i-1}}$.

\subsection{Other Concentration Inequalities for Polynomials}

  Another line of attack on understanding the concentration of  polynomials is to use the structure of polynomials and some smoothness parameters analogous to the partial derivatives or our parameters $\mu_r$. Many of these known inequalities provide tight upper and lower bounds for moments but involve hard to estimate smoothness parameters. In the case of Gaussian random variables tight concentration bounds for polynomials of independent random variables were obtained by Hanson and Wright  \cite{HW} for  and $q=2$ and  Borell \cite{B}, Arcones and Gine \cite{AG} for $q\ge 3$. In the case of Rademacher random variables analogous results but based on different methods were obtained by Talagrand \cite{T} for $q=2$ and Boucheron et al. \cite{ BBLM} for $q\ge 3$. Adamczak \cite{A} proved a concentration  inequality for general functions of a general class of random variables.
  All these results except \cite{HW} and \cite{T} (i.e. the case of quadratic polynomials) use parameters that involve expectations of suprema of certain empirical processes that are in general not easy to estimate which limits applicability of these inequalities.

  Another interesting class of inequalities was obtained by using the so-called "needle decomposition method" in the field of Geometric Functional Analysis. It is a rich research area and we refer the reader to the survey paper by Nazarov, Sodin and Volberg \cite{NSV}. An interesting moment inequality which seem to generalize and tighten many previously known inequalities in this area was shown by Carbery and Wright \cite{CW} (see Theorem 7). It implies the following concentration inequality via application of Markov's inequality
  \begin{theorem}\label{generalTheorem1}
Consider a  degree $q$ polynomial $f(Y)=f(Y_1,\dots,Y_n)$. Assume that random variables  $Y_1,\dots,Y_n$ are distributed according to some log-concave measure in $R^n$ (i.e. they are not necessarily independent).  Then
$$Pr[|f(Y)-\E[f(Y)]|\ge \lambda] \le  e^2\cdot e^{-\left(\frac{\lambda}{R\sqrt{Var[f(Y)]}}\right)^{1/q}}$$
for some absolute constant $R>0$.
  \end{theorem}

  On one side this inequality is extremely general and allows to study such processes as sampling a point uniformly from the interior of a polytope in $R^n$. On the other side, due to its generality this inequality is weaker than ours even in  a simple case of the sum of $n$ independent exponential random variables ($q=1$). In this case our concentration inequality gives Chernoff bounds like estimates while Theorem \ref{generalTheorem1} provides a much weaker bound. Another drawback of the Theorem \ref{generalTheorem1} that it does not handle discrete distributions.

\subsection{Paper Outline}

We now outline the rest of this paper.

In Section \ref{sec:centered} we state and prove several lemmas about the moments of ``centered'' polynomials that form the heart of our results. In Section \ref{ProofMainGen} we extend these lemmas to moments of arbitrary polynomials. In Section \ref{ProofMain} we use these Lemmas to prove our main Theorem \ref{main1}. In Section \ref{sec:counting} we prove a counting lemma used in Section \ref{sec:centered}. We prove our permanent Theorems \ref{perm} and \ref{perm1} in Section \ref{sec:perm}. We prove our lower bound Theorem \ref{thm:LB} in Section \ref{tight}. We conclude with examples of moment bounded random variables in Section \ref{sec:exampleMom}.

In Appendix \ref{sec:linear} we prove a special case of our main result: the linear case $q=1$, i.e.\ concentration of a sum of independent moment-bounded random variables. This linear case of our theorem is not new, but the proof nicely illustrates many of our techniques with minimal technical complications. The interested reader may find it helpful to study the special cases in Section \ref{sec:methods} and Appendix \ref{sec:linear} before reading the main body of this paper.

\section{Moment Lemma for Centered Polynomials}\label{sec:centered}

The proof of the Theorem \ref{main1} will follow from the application of the Markov's inequality to the upper bound on the $k$-th moment of the polynomial in question. The first step is to look at moments of ``centered'' polynomials that replace $Y_v^{\tau_{hv}}$ with $\left(Y_v^{\tau_{hv}}-\E\left[Y_v^{\tau_{hv}}\right]\right)$.
For simplicity the heart of our analysis will assume that all coefficients $w_h$ are non-negative; negative coefficients will return in Section \ref{ProofMainGen}.

\begin{lemma}[Initial Moment Lemma]\label{lem:momentPrelim}
 We are given a hypergraph $H = ([n], {\cal H})$, $n$ independent moment bounded random variables $Y=(Y_1,\dots, Y_n)$ with the same parameter $L$ and a polynomial
 \[
 g(y)= \sum_{h\in {\cal  H}} w_h\prod_{v\in h}\left(y_v^{\tau_{hv}}-\E\left[Y_v^{\tau_{hv}}\right]\right)
 \]
  with nonnegative coefficients $w_h \ge 0$ such that every monomial (or hyperedge) $h\in {\cal  H}$ has cardinality exactly $\eta$, total power exactly $q$ and maximal power upper bounded by $\Gamma$, i.e.\ $q(h)=q$, $\eta(h)=\eta$ and $\Gamma\ge\max_{v\in h} \tau_{hv}$. It follows that for any integer $k \ge 1$ we have
 \begin{eqnarray}\label{inequality44}
\left|\E\left[g(Y)^k\right]\right| & \le & \max_{{\bar \nu}} \left\{R_2^{qk} L^{qk-\Delta}\cdot \Gamma^{qk-\Delta}\cdot k^{qk-(q-1)\nu_0-\Delta}\cdot \left(\prod_{t=0}^{q}\mu_t^{\nu_t} \right) \right\}
\end{eqnarray}
where $R_2\ge 1$ is some absolute constant, $\Delta = \sum_{t=0}^{q}(q-t)\nu_t$, and the maximum is over all non-negative integers $\nu_t$, $0 \le t \le q$ satisfying $\nu_0 \le \nu_q$, $\sum_{t=0}^q\nu_t=k$ and $qk-(q-1)\nu_0-\Delta\ge 1$ (note also that $\mu_t$ are defined according to (\ref{def:mu}), i.e. they depend on original (not centered) random variables $Y_v$ for $v\in [n]$).
\end{lemma}

\begin{proof}
Fix hypergraph $H=([n],{\cal H})$, random variables $Y=(Y_1,\dots,Y_n)$, non-negative weights $\{w_h\}_{h \in {\cal H}}$, integer $k\ge 1$, cardinality $\eta$ and total power $q$.
 Without loss of generality we assume that ${\cal H}$ is the complete hypergraph (setting additional edge weights to 0 as needed), i.e.\ ${\cal H}$ includes every possible hyperedge over vertex set $[n]$ with cardinality $\eta$, total power $q$, and maximal power at most $\Gamma$. A \emph{labeled} hypergraph $G=({\cal V}(G), {\cal H}(G))$ consists of a set of vertices ${\cal V}(G)$ and a \emph{sequence} of $k$ (not necessarily distinct) hyperedges ${\cal H}(G)=h_1,\dots,h_k$. In other words a labeled hypergraph is a hypergraph whose $k$ hyperedges are given unique labels from $[k]$.
We write e.g.\ $\prod_{h \in {\cal H}(G)} w_h$ as a shorthand for $\prod_{i=1}^k w_{h_i}$ where ${\cal H}(G)=h_1,\dots,h_k$; in particular duplicate hyperedges count multiple times in such a product.

Consider the sequence of hyperedges $h_1,\dots,h_k\in {\cal  H}$ from our original hypergraph $H$. These hyperedges define a labeled hypergraph $H(h_1,\dots,h_k)$ with vertex set $\cup_{i=1}^k {\cal V}(h_i)$ and hyperedge sequence $h_1,\dots,h_k$.
Note that the vertices of $H(h_1,\dots,h_k)$ are labeled by the indices from $[n]$ and the edges are labeled by the indices from $[k]$.
Note also that some hyperedges in $H(h_1,\dots,h_k)$ could span the same set of vertices and have the same power vector, i.e.\ they are multiple copies of the same hyperedge in the original hypergraph $H$.
 Let ${\cal P}(H,k)$ be the set of all such edge and vertex labeled hypergraphs that can be generated by any $k$ hyperedges from $H$.
 We say that the \emph{degree} of a vertex (in a hypergraph) is the number of hyperedges it appears in.
Let ${\cal P}_2(H,k)\subseteq {\cal P}(H,k)$ be the set of such labeled hypergraphs where each vertex has degree at least two.
We split the whole proof into more digestible pieces by subsections.

\subsection{Changing the vertex labeling}

In this section we will show how to transform the formula for the $k$-th moment to have the summation over the hypergraphs that have its own set of labels instead of being labeled by the set $[n]$. Let $X_{hv}=Y_v^{\tau_{hv}} - \Eb{Y_v^{\tau_{hv}}}$ for $h\in {\cal  H}$ and $v\in h$.
By linearity of expectation, independence of random variables $X_{hv}$ for different vertices $v\in {\cal V}$ and definition of ${\cal P}(H,k)$ we obtain
\begin{eqnarray}
\left|\E\left[g(Y)^k\right]\right|
&=&\left|\sum_{h_1,\dots,h_k\in {\cal  H}}\E\left[\prod_{i=1}^k\left(w_{h_i}\prod_{v\in \ve{h_i}}X_{h_iv}\right)\right]\right|\nonumber \\
&=&\left|\sum_{G\in  {\cal P}(H,k)}\E\left[\prod_{h\in {\cal H}(G)}\left(w_{h}\prod_{v\in \ve{h}}X_{hv}\right)\right]\right|\nonumber \\
&=&\left|\sum_{G\in  {\cal P}(H,k)}\left(\prod_{h\in {\cal H}(G)}w_{h} \right)\left(\prod_{v\in {\cal V}(G)}\E\left[\prod_{h\in {\cal H}(G)|v\in \ve{h}}X_{hv}\right]\right)\right| \nonumber \\
&=&\left|\sum_{G\in  {\cal P}_2(H,k)}\left(\prod_{h\in {\cal H}(G)}w_{h} \right)\left(\prod_{v\in {\cal V}(G)}\E\left[\prod_{h\in {\cal H}(G)|v\in \ve{h}}X_{hv}\right]\right) \right| \nonumber \\
&\le &\sum_{G\in  {\cal P}_2(H,k)}\left(\prod_{h\in {\cal H}(G)}w_{h} \right)\left(\prod_{v\in {\cal V}(G)}\left|\E\left[\prod_{h\in {\cal H}(G)|v\in \ve{h}}X_{hv}\right]\right|\right) , \label{inequality1}
\end{eqnarray}
where the last equality follows from the fact that $\E[X_{hv}]=0$  for all $h\in {\cal H}$ and $v\in \ve{h}$.
Below we will use the notation $\Lambda_v(G)=\left|\E\left[\prod_{h\in {\cal H}(G)|v\in \ve{h}}X_{hv}\right]\right|$.

Note that a labeled hypergraph $G\in {\cal P}_2(H,k)$ could have the number of vertices ranging from $\eta$ up to $k\eta/2$ since every vertex has degree at least two. For $q$, $\eta$ and $\Gamma$ clear from context, let ${\cal S}_2(k,\ell)$ be the set of labeled hypergraphs
with vertex set $[\ell]$ having $k$ hyperedges such that each hyperedge has cardinality exactly $\eta$, total power $q$, maximal power $\le \Gamma$, and every vertex has degree at least 2. For each hypergraph $G\in {\cal S}_2(k,\ell)$ the vertices are labeled by the indices from the set $[\ell]$ and the edges are labeled by the indices from the set $[k]$.
  Let $M(S)$ for $S\subseteq [\ell]$ be the set of all possible injective functions $\pi:S\rightarrow [n]$, in particular
  $M([\ell])$ is the set of all possible injective functions $\pi:[\ell]\rightarrow [n]$. We will use the notation $\pi(h)$ for a copy of hyperedge $h = ({\cal V}(h), \tau(h))\in {\cal H}(G)$ with its vertices relabeled by injective function $\pi$, i.e.\ ${\cal V}(\pi(h))=\{\pi(v) : v \in {\cal V}(h)\}$ and $\tau_{\pi(h),\pi(v)} = \tau_{hv}$. Analogously we will use notation $\pi(G)$ to denote the graph $G$ with vertices re-labeled according to function $\pi$. We claim that
\begin{eqnarray}
\lefteqn{\sum_{G\in  {\cal P}_2(H,k)}\left(\prod_{h\in {\cal H}(G)}w_{h} \right)\left(\prod_{v\in {\cal V}(G)}\Lambda_v(G)\right)} \nonumber \\
&=& \sum_{\ell=\eta}^{k\eta/2}\frac{1}{\ell!}\sum_{G'\in {\cal S}_2(k,\ell)}\sum_{\pi\in M([\ell])} \left(\prod_{h\in {\cal H}(G')}w_{\pi(h)} \right)\left(\prod_{u\in {\cal V}(G')}\Lambda_{\pi(u)}(\pi(G'))\right). \label{inequality2}
\end{eqnarray}
Indeed, every labeled hypergraph $G=({\cal V}(G),{\cal H}(G))\in  {\cal P}_2(H,k)$ on $\ell$ vertices has $\ell!$ labeled hypergraphs $G'=({\cal V}(G'),{\cal H}(G'))\in {\cal S}_2(k,\ell)$ that differ from $G$ by vertex labellings only. Each of those hypergraphs has one corresponding mapping $\pi$ that maps its $\ell$ vertex labels into vertex labels of hypergraph $G\in {\cal P}_2(H,k)$.

Then, combining (\ref{inequality1}) and (\ref{inequality2}) we obtain
\begin{eqnarray}\label{inequality3}
\left|\E\left[g(Y)^k\right]\right|\le \sum_{\ell=\eta}^{k\eta/2}\frac{1}{\ell!}\sum_{G'\in {\cal S}_2(k,\ell)} \sum_{\pi\in M([\ell])} \left(\prod_{h\in {\cal H}(G')}w_{\pi(h)} \right)\left(\prod_{u\in {\cal V}(G')}\Lambda_{\pi(u)}(\pi(G')) \right).
\end{eqnarray}

\subsection{Estimating the term for each hypergraph $G'$}\label{sec:eachGraph}

We now fix integer $\ell$ and labeled hypergraph $G'\in {\cal S}_2(k,\ell)$.
Let $c$ be the number of \emph{connected components} in $G'$, i.e.\ $c$ is a maximal number such that the vertex set ${\cal V}(G')$ can be partitioned into $c$ parts ${\cal V}_1,\dots,{\cal V}_c$ such that for each hyperedge $h\in {\cal H}(G')$ and any $j\in [c]$ if $\ve{h}\cap {\cal V}_j\neq \emptyset$ then $\ve{h}\subseteq {\cal V}_j$. Intuitively, we can split the vertex set of $G'$ into $c$  components such that there are no hyperedges that have vertices in two or more components.
For each vertex $v\in {\cal V}(G')$, we define $D_v=\sum_{h\in {\cal H}(G')|v\in \ve{h}}\tau_{hv}$  to be the sum of all the powers that correspond to the vertex $v$ and hyperedges that are incident to $v$. We call $D_v$ the \emph{total power} of $v$. Let $d_v$ denote the number of hyperedges $h \in {\cal H}(G')$ with $v \in \ve{h}$. We will call $d_v$ the \emph{degree} of the vertex $v$. By definitions
$\sum_{v\in {\cal V}(G')}d_v=\eta k$, $\sum_{v\in {\cal V}(G')}D_v=q k$ and $d_v\ge 2$ for all $v\in {\cal V}(G')$.

We consider a certain canonical ordering $h^{(1)},\dots,h^{(k)}$ of the hyperedges in ${\cal H}(G')$ that will be specified later in Lemma \ref{lem:ordering}. (This ordering is distinct from and should not be confused with the ordering of the hyperedges inherent in a labeled hypergraph.)
We iteratively remove hyperedges from the hypergraph $G'$ in this order. Let $G'_s=({\cal V}_s',{\cal H}_s')$ be the hypergraph defined by the hyperedges ${\cal H}_s'=h^{(s)},\dots,h^{(k)}$ and vertex set ${\cal V}_s'=\cup_{h \in {\cal H}_s'} \ve{h}$. In particular $G'_1$ is identical to $G'$ except for the order of the hyperedges.
 Let $V_s$ be the vertices of the hyperedge $h^{(s)}$ that have degree one in the hypergraph $G'_s$, i.e.\ ${\cal V}_{s+1}'={\cal V}_s'\setminus V_s$.  By definition, $0\le |V_s|\le \eta$.   For each vertex $v\in V_s$, let $\delta_v=\tau_{h^{(s)}v}$, i.e.\ $\delta_v$ is the power corresponding to the vertex $v$ and the last hyperedge in the defined order that is incident to $v$. We call $\delta_v$ the \emph{last power} of $v$. Intuitively, we delete edges in the order $h^{(1)},\dots,h^{(k)}$. We also delete all vertices that become isolated. Then $V_s$ is the set of vertices that get deleted during step $s$ of this process.

Recall $[\ell]= {\cal V}(G')$ and ${\cal V}_s'={\cal V}(G')\setminus \cup_{t=1}^{s-1} V_t$ for $s=1,\dots,k$.    Then
\begin{eqnarray}
&&\sum_{\pi\in M({\cal V}_s')} \left(\prod_{h\in {\cal H}_s'}w_{\pi(h)} \right)\left(\prod_{v\in {\cal V}_s'}\Lambda_{\pi(v)}(\pi(G'))\right) = \nonumber \\
&& \sum_{\pi'\in M({\cal V}_{s+1}')} \sum_{\substack{\pi\in M({\cal V}_s') \\ \text{s.t. } \pi \text{ extends } \pi'}} \left(\prod_{h\in {\cal H}_{s+1}'} w_{\pi(h)} \right)\left(\prod_{v\in {\cal V}_{s+1}'}\Lambda_{\pi(v)}(\pi(G'))\right)\ \left(w_{\pi(h^{(s)})}\prod_{v\in V_s} \Lambda_{\pi(v)} (\pi(G'))\right) = \nonumber \\
&& \sum_{\pi'\in M({\cal V}_{s+1}')}\left[ \left(\prod_{h\in {\cal H}_{s+1}'}w_{\pi'(h)} \right)\left(\prod_{v\in {\cal V}_{s+1}'}\Lambda_{\pi'(v)}(\pi(G'))\right) \sum_{\substack{\pi\in M({\cal V}_s') \\ \text{s.t. } \pi \text{ extends } \pi'}}\left(w_{\pi(h^{(s)})}\prod_{v\in V_s} \Lambda_{\pi(v)}(\pi(G')) \right)\right]  \nonumber
\end{eqnarray}
where we say that $\pi$ \emph{extends} $\pi'$ if $\pi(v)=\pi'(v)$ for every $v$ in the domain of $\pi'$. By Lemma \ref{lem:momBound} which is an implication of the moment boundness of random variables (Section \ref{sec:technical}) we have
\begin{eqnarray*}
\Lambda_{\pi(v)}(\pi(G'))&\le& \frac{2^{d_{\pi(v)}}L^{D_{\pi(v)}-\tau_{\pi\left(h^{(s)}\right)\pi(v)}}\cdot D_{\pi(v)}!\cdot
\E\left[\left|Y_{\pi(v)}^{\tau_{\pi\left(h^{(s)}\right)\pi(v)}}\right|\right]}{\tau_{\pi\left(h^{(s)}\right)\pi(v)}!} \\
&=&\frac{2^{d_v}L^{D_{v}-\delta_v}\cdot D_{v}!\cdot \E\left[\left|Y_{\pi(v)}^{\tau_{\pi\left(h^{(s)}\right)\pi(v)}}\right|\right]}{\delta_v!}
\end{eqnarray*}
since $\delta_v=\tau_{\pi\left((h^{(s)}\right)\pi(v)}$ (the vertex degrees and powers do not depend on vertex labeling).
Therefore,
\begin{eqnarray*}
\sum_{\substack{\pi\in M({\cal V}_s') \\ \text{s.t. } \pi \text{ extends } \pi'}}&& \left(w_{\pi(h^{(s)})}\prod_{v\in V_s} \Lambda_{\pi(v)}(\pi(G')) \right)
\le \\
&&\prod_{v\in V_s} \frac{2^{d_v} L^{D_{v}-\delta_v}\cdot D_{v}!}{\delta_v!}
\cdot
\sum_{\substack{\pi\in M({\cal V}_s') \\ \text{s.t. } \pi \text{ extends } \pi'}}\left(w_{\pi\left(h^{(s)}\right)}\prod_{v\in V_s}
\E\left[\left|Y_{\pi(v)}^{\tau_{\pi\left(h^{(s)}\right)\pi(v)}}\right|\right]\right).
\end{eqnarray*}
 We now group the sum over $\pi$ by the value of $\pi(h^{(s)}) \equiv h \in {\cal H}$. Note that for any fixed mapping $\pi'\in M({\cal V}_{s+1}')$ there are at most $|V_s|!$ possible mappings $\pi \in M({\cal V}_s')$ that extend $\pi'$ and map the vertex labels of hyperedge $h^{(s)}\in G'$ into vertex labels of the hyperedge $h\in {\cal H}$.
 Let $h^{(s)} \setminus V_s$ denote $h^{(s)}$ but with vertices $V_s$ removed, i.e.\ ${\cal V}(h^{(s)} \setminus V_s) = {\cal V}(h^{(s)}) \setminus V_s$ and $\tau_{h^{(s)} \setminus V_s,v} = \tau_{h^{(s)}v}$ for all $v \in {\cal V}(h^{(s)} \setminus V_s)$.
  Let $h'=\pi'(h^{(s)} \setminus V_s)$, which is the portion of $\pi(h^{(s)})$ that is fixed by $\pi'$.
 Recall, we write $h \extends h'$ if ${\cal V}(h) \supseteq {\cal V}(h')$ and $\tau_{hv}=\tau_{h'v}$ for all $v\in {\cal V}(h')$. Also recall the notation $\delta_v=\tau_{h^{(s)}v}$ for $v\in V_s$. Then
\begin{eqnarray*}
\sum_{\substack{\pi\in M({\cal V}_s') \\ \text{s.t. } \pi \text{ extends } \pi'}} w_{\pi(h^{(s)})}\prod_{v\in V_s} \E\left[\left|Y_{\pi(v)}^{\tau_{\pi\left(h^{(s)}\right)\pi(v)}}\right|\right]
&\le& |V_s|! \sum_{h \in {\cal H}|h \extends h'} w_{h} \prod_{u\in \ve{h} \setminus \ve{h'}} \E\left[\left|Y_u^{\tau_{hu}}\right|\right]\\
&\le& |V_s|! \max_{h'|q(h')=q-\sum_{v\in V_s} \delta_v}\left\{\sum_{h \in {\cal H}|h \extends h'} w_h\prod_{v\in \ve{h}\setminus \ve{h'}}\E[|Y_v^{\tau_{hv}}|]\right\}\\
& =& |V_s|! \mu_{q-\sum_{v\in V_s}\delta_v}.
\end{eqnarray*}
We repeat the argument for $s=1,\dots,k$. In the end we obtain
\begin{eqnarray} \label{estimate}
\sum_{\pi\in M([\ell])} \left(\prod_{h\in {\cal H}(\pi(G'))}w_{\pi(h)} \right)\left(\prod_{v\in {\cal V}(G')}\Lambda_{\pi(v)}(\pi(G'))\right) &=&  \nonumber \\
\sum_{\pi\in M([\ell])} \left(\prod_{h\in {\cal H}_1'}w_{\pi(h)} \right)\left(\prod_{v\in {\cal V}_1'}\Lambda_{\pi(v)}(G')\right)&\le& \nonumber \\
 \left(\prod_{v\in {\cal V}(G')} \frac{2^{d_v} L^{D_{v}-\delta_v}\cdot D_{v}!}{\delta_v!}\right)\prod_{s=1}^k \left(|V_s|!  \mu_{q-\sum_{v\in V_s}\delta_v} \right)&=& \nonumber\\
 2^{\eta k} L^{qk-\Delta}\left(\prod_{v\in {\cal V}(G')} \frac{D_{v}!}{\delta_v!}\right)\left(\prod_{s=1}^k |V_s|!  \right)\prod_{t=0}^{q}\mu_t^{\nu_t} &\le&  \nonumber\\
2^{\eta k} L^{qk-\Delta}\left(\prod_{v\in {\cal V}(G')} \frac{D_{v}!}{\delta_v!}\right) \eta^\ell \cdot \prod_{t=0}^{q}\mu_t^{\nu_t}
\end{eqnarray}
where $\nu_t$ is the number of indices $s=1,\dots,k$ with $q-\sum_{v\in V_s}\delta_v=t$, $\mu_t = \mu_t(w,Y)$, and $\Delta=\sum_{v\in {\cal V}(G')}\delta_v$. In the last inequality we used the fact that $\sum_{s=1}^k|V_s|=\ell$ and $|V_s|\le \eta$.
The quantities $\nu_t$ must satisfy the equality $\sum_{t=0}^{q} (q-t)\nu_t = \sum_{v\in {\cal V}(G')}\delta_v=\Delta$.

The flexibility in the choice of the ordering $h^{(1)},\dots,h^{(k)}$ affects the quality of the bound (\ref{estimate}) via its influence on the $\nu_t$, $\delta_v$ and $\Delta$. We focus on minimizing $\nu_0$, which intuitively makes sense as $\mu_0$ is often much larger than the other $\mu_t$. The last hyperedge in each of the $c$ connected components must contribute to $\nu_0$, so $\nu_0 \ge c$. It turns out that equality $\nu_0=c$ is achievable; the intuition is to pick an ordering that never splits a connected component of $G'_s$ into several components. We defer the proof that such an ordering exists to Lemma \ref{lem:ordering} in Section \ref{sec:ordering}.
We also know that $\nu_q\ge c$ because the first hyperedge in each connected component is incident to vertices of degree two or more only and therefore contributes to $\nu_q$.

\subsection{Using the Counting Lemma}\label{sec:useCount}
We assume that each hypergraph in ${\cal S}_2(k,\ell)$ has an associated canonical ordering of hyperedges, formally defined in Lemma \ref{lem:ordering}. This canonical ordering specifies the last powers for all vertices in the hypergraph and the values $\nu_t$ for $t=0,\dots, q$.

We decompose ${\cal S}_2(k,\ell)$ as ${\cal S}_2(k,\ell) = \bigcup_{c,\bar{d} \ge \bar{2}, \bar{D}, \bar{\delta}}{\cal S}(k,\ell, c, \bar{d}, \bar{D}, \bar{\delta})$ where $\bar{2}$ is a vector of $\ell$ twos and ${\cal S}(k, \ell,c,\bar{d},\bar{D},\bar{\delta})$ is the set of vertex and hyperedge labeled  hypergraphs with vertex set $[\ell]$ and $k$ hyperedges such that each hyperedge has cardinality $\eta$, total power $q$,  maximal power $\le \Gamma$, the number of connected components is $c$, the degree vector is $\bar{d}$, the total power vector is $\bar{D}$, and $\bar{\delta}$ is the vector of last powers (corresponding to the canonical ordering). Note that ${\cal S}(k, \ell, c, \bar{d}, \bar{D}, \bar{\delta})$ depends on $q$, $\eta$ and $\Gamma$ as well.
Let ${\bar \nu}=(\nu_0,\dots,\nu_q)$. Combining, (\ref{inequality3}) and (\ref{estimate}) we obtain
\begin{eqnarray}
&& \left|\E\left[g(Y)^k\right]\right|
 \le  \sum_{\ell=\eta}^{k\eta/2}\frac{1}{\ell!}\sum_{c=1}^{\ell/\eta}\sum_{\bar{d}\ge \bar{2},\bar{D},\bar{\delta}} \sum_{G'\in {\cal S}(k,\ell,c,\bar{d},\bar{D},\bar{\delta})} 2^{\eta k} L^{qk-\Delta} \left(\prod_{t=0}^{q}\mu_t^{\nu_t} \right) \eta^{\ell} \prod_{v \in [\ell]} \frac{D_{v}!}{\delta_v!} \nonumber \\
&& \le  \max_{\ell,c,\bar{d}\ge \bar{2}, \bar{D},\bar{\delta},{\bar \nu}}\left\{ \frac{k\eta}{2}\cdot \frac{1}{\ell!} \cdot\frac{\ell}{\eta}\cdot (2^{q k + \ell})^3 \cdot  |{\cal S}(k,\ell,c,\bar{d},\bar{D},\bar{\delta})| \cdot 2^{\eta k} L^{qk-\Delta}\left(\prod_{t=0}^{q}\mu_t^{\nu_t} \right)\eta^{\ell} \prod_v \frac{D_{v}!}{\delta_v!} \right\} \nonumber \\
&& \le   \max_{\ell,c,\bar{d}\ge \bar{2}, \bar{D},\bar{\delta},{\bar \nu}} \Bigg\{\frac{k\ell}{2\cdot \ell !}\cdot 2^{3(qk + \ell)} \cdot 2^{\eta k} L^{qk-\Delta} \cdot \eta^{\ell} \cdot \underbrace{R_0^{qk}\cdot \Gamma^{qk-\Delta-\ell}\cdot k^{qk-(q-1)c-\Delta+\ell}}_{\substack{|{\cal S}(k,\ell,c,\bar{d},\bar{D},\bar{\delta})|\prod_v \frac{D_{v}!}{\delta_v!} ~\le~ \text{this}\\ \text{by counting Lemma \ref{MainCount} (Section \ref{sec:counting})}}}\cdot \left(\prod_{t=0}^{q}\mu_t^{\nu_t} \right) \Bigg\}\nonumber
\end{eqnarray}
where the maximum over $\bar{\nu}$ is over $\nu_0,\dots,\nu_q \ge 0$ with $c=\nu_0 \le \nu_q$, $\sum_{t=0}^q \nu_t = k$ and $\sum_{t=0}^q (q-t)\nu_t = \Delta$. Also the integers $\nu_0,\dots,\nu_q$ must satisfy the inequality
$(q-1)k-\Delta\ge (q-2)\nu_0$ by Corollary \ref{Cor_Ineq}.
The second inequality follows from the fact that the total number of feasible total power vectors $\bar{D}$ is at most $2^{q k+\ell}$ ($q k$ is the sum of all the powers and we need to compute the total number of partitions of the array with $q k$ entries into $\ell$ possible groups of consecutive entries which is ${q k+\ell-1 \choose \ell-1}$), and similarly the number of vectors $\bar{d}$ and $\bar{\delta}$ can also be upper-bounded by $2^{q k+\ell}$.
We substitute $\nu_0$ for $c$, and remove the unreferenced variables $c$, $\bar{d}$, $\bar{D}$ and $\bar{\delta}$ from the maximum. The maximum over $\bar{\nu}$ is now over $\nu_0,\dots,\nu_q \ge 0$ with $\nu_0 \le \nu_q$, $\sum_{t=0}^q \nu_t = k$ such that $(q-1)k-\Delta\ge (q-2)\nu_0$. We continue
\begin{eqnarray}
\left|\E\left[g(Y)^k\right]\right|
& \le &  \max_{\ell,{\bar \nu}} \left\{\frac{k\ell}{2\cdot \ell !}\cdot 2^{3(qk + \ell)} \cdot 2^{\eta k} L^{qk-\Delta} \cdot \eta^{\ell} \cdot R_0^{qk}\cdot \Gamma^{qk-\Delta-\ell}\cdot k^{qk-(q-1)\nu_0-\Delta+\ell}\cdot \left(\prod_{t=0}^{q}\mu_t^{\nu_t} \right) \right\}\nonumber \\
& \le &  \max_{\ell,{\bar \nu}} \left\{R_1^{qk} \cdot \left(\frac{\eta}{\ell}\right)^{\ell}\cdot L^{qk-\Delta}\cdot \Gamma^{qk-\Delta}\cdot k^{qk-(q-1)\nu_0-\Delta+\ell}\cdot \left(\prod_{t=0}^{q}\mu_t^{\nu_t} \right) \right\}\nonumber \\
& = &  \max_{\ell,{\bar \nu}} \left\{R_1^{qk} \cdot \left(\frac{\eta k}{\ell}\right)^{\ell} \cdot L^{qk-\Delta}\cdot \Gamma^{qk-\Delta}\cdot k^{qk-(q-1)\nu_0-\Delta}\cdot \left(\prod_{t=0}^{q}\mu_t^{\nu_t} \right) \right\}\nonumber \\
& \le &  \max_{{\bar \nu}} \left\{R_2^{qk} L^{qk-\Delta}\cdot \Gamma^{qk-\Delta}\cdot k^{qk-(q-1)\nu_0-\Delta}\cdot \left(\prod_{t=0}^{q}\mu_t^{\nu_t} \right) \right\}\label{UsingCount}
\end{eqnarray}
where $R_0<R_1<R_2$ are some absolute constants, the second inequality uses the facts that $\ell!\ge (\ell/e)^\ell$ and $\Gamma^{-\ell} \le 1$, and the last inequality is implied by the fact that
\[
\left(\frac{k\eta}{\ell}\right)^{\ell} \le \max_{x>0} \left(\frac{k\eta}{x}\right)^{x} = e^{k\eta/e}.
\]
Inequality (\ref{UsingCount}) is precisely the inequality (\ref{inequality44}) that we needed to prove. Note that the inequality $(q-1)k-\Delta\ge (q-2)\nu_0$ implies
$qk-(q-1)\nu_0-\Delta\ge k-\nu_0=k-c\ge c\ge 1$ (we use the fact that each connected component has at least two hyperedges).
\end{proof}

\subsection{Intermediate moment lemma}

\begin{lemma}[Intermediate Moment Lemma]\label{lem:moment}
We are given $n$ independent moment bounded random variables $Y=(Y_1,\dots, Y_n)$ with the same parameter $L$ and a general polynomial $f(x)$ with nonnegative coefficients such that every monomial (or hyperedge) $h\in {\cal  H}$ has  exactly $\eta$ variables, total power exactly $q$ and power of any variable upper bounded by $\Gamma$, i.e.\ $q(h)=q$, $\eta(h)=\eta$ and $\Gamma=\max_{h\in {\cal  H}, v\in h} \tau_{hv}$. Then
 \begin{eqnarray}\label{inequality4}
\left|\E\left[g(Y)^k\right]\right| & \le & \max \left\{ \left(\sqrt{kR_3^q\Gamma^q L^q \mu_q\mu_0}\right)^{k}, \max_{t\in [q]}(k^{t}R_3^qL^t\Gamma^t\mu_t)^{k}\right\}.
\end{eqnarray}
where $R_3\ge 1$ is some absolute constant, $g(Y)$ is a polynomial of centered random variables defined in Lemma \ref{lem:momentPrelim} and $[q]=\{1,\dots,q\}$.
\end{lemma}

\begin{proof}
We apply Lemma \ref{lem:momentPrelim}.
Since $\sum_{t=0}^q\nu_t=k$ and $\sum_{t=0}^{q}(q-t)\nu_t =\Delta$ we have,
\begin{eqnarray*}
q(\nu_q-\nu_0)+\nu_0+\sum_{t=1}^{q-1}t\nu_t&=&-(q-1)\nu_0+\sum_{t=0}^{q}t\nu_t\\
&=&q\left(k-\sum_{t=0}^q\nu_t\right)-(q-1)\nu_0+\sum_{t=0}^{q}t\nu_t\\
&=&qk-(q-1)\nu_0-\sum_{t=0}^{q}(q-t)\nu_t\\
&=&qk-(q-1)\nu_0-\Delta,\\
q(\nu_q-\nu_0)+q\nu_0+\sum_{t=1}^{q-1}t\nu_t&=&qk-\Delta.
\end{eqnarray*}
Therefore,
\begin{eqnarray}
\max_{{\bar \nu}} \left\{R_2^{qk} L^{qk-\Delta}\cdot \Gamma^{qk-\Delta}\cdot k^{qk-(q-1)\nu_0-\Delta}\cdot \left(\prod_{t=0}^{q}\mu_t^{\nu_t} \right) \right\}\nonumber \\
=  \max_{{\bar \nu}} \left\{(k^qR_2^q\Gamma^{q}L^q\mu_q)^{\nu_q-\nu_0}\cdot (k R_2^{2q}\Gamma^{q}L^q\mu_q\mu_0)^{\nu_0}\cdot \prod_{t=1}^{q-1}(k^{t}R_2^q\Gamma^tL^t\mu_t)^{\nu_t}\right\}\nonumber \\
\le  \max_{{\bar \nu}} \left\{(k^qR_3^q\Gamma^{q}L^q\mu_q)^{\nu_q-\nu_0}\cdot (k R_3^{q}\Gamma^{q}L^q\mu_q\mu_0)^{\nu_0}\cdot \prod_{t=1}^{q-1}(k^{t}R_3^q\Gamma^tL^t\mu_t)^{\nu_t}\right\}\nonumber
\end{eqnarray}
for the absolute constant $R_3=R_2^2$. Using the facts that $\sum_{t=0}^q\nu_t=k$ and $\nu_q\ge \nu_0$ again, we derive,
\begin{eqnarray*}
\left|\E\left[g(Y)^k\right]\right| & \le & \max_{{\bar \nu}} \left\{(k^qR_3^q\Gamma^{q}L^q\mu_q)^{\nu_q-\nu_0}\cdot \left(\sqrt{k R_3^{q}\Gamma^{q}L^q\mu_q\mu_0}\right)^{2\nu_0}\cdot \prod_{t=1}^{q-1}(k^{t}R_3^q\Gamma^tL^t\mu_t)^{\nu_t}\right\}\\
&\le& \max \left\{\left(k^qR_3^q\Gamma^{q}L^q \mu_q\right)^{k},\left(\sqrt{kR_3^{q}\Gamma^{q}L^q \mu_q\mu_0}\right)^{k}, \max_{t\in [q-1]}\left(k^{t}R_3^q\Gamma^tL^t \mu_t\right)^{k}\right\}  \nonumber \\
&=&  \max \left\{ \left(\sqrt{kR_3^{q}\Gamma^{q}L^q \mu_q\mu_0}\right)^{k}, \max_{t\in [q]}\left(k^{t}R_3^q\Gamma^tL^t\mu_t\right)^{k}\right\}.
\end{eqnarray*}
\end{proof}

\subsection{Three Technical Lemmas}\label{sec:ordering}\label{sec:technical}

In this section we prove  technical lemmas that were used in the proof of Lemma \ref{lem:momentPrelim}.

\begin{lemma}\label{lem:momBound}
For any moment bounded random variable $Z$ with parameter $L$, integer $k\ge 1$, set $S\subseteq [k]$ and a collection of positive integer powers $d_t$ for $t\in S$, the following inequality holds:
$$\left|\E\left[\prod_{t\in S} \left(Z^t\right)^{d_t}\right]\right|
\le \min_{t\in S}\left\{\frac{L^{D-t}\cdot D!\cdot \E\left[\left|Z\right|^t\right]}{t!}\right\}$$
and
$$\left|\E\left[\prod_{t\in S} \left(Z^t-\E\left[Z^t\right]\right)^{d_t}\right]\right|
\le \min_{t\in S}\left\{  \frac{2^d L^{D-t}\cdot D!\cdot \E\left[\left|Z\right|^t\right]}{t!}\right\}$$
where $D=\sum_{t\in S}td_t$ and $d=\sum_{t\in S}d_t$.
\end{lemma}
\begin{proof}
To prove the first inequality note that for any $\tau \in S$ by Jensen's inequality we have
\begin{align}
\left|\E\left[\prod_{t\in S} \left(Z^t\right)^{d_t}\right]\right| &= |\Eb{Z^D}| \le \Eb{|Z|^D}  \nonumber\\
&\le \frac{L^{D-\tau} D!}{\tau!}\E\left[\left|Z\right|^{\tau}\right] \label{eqn:applyMomBound}
\end{align}
where the final inequality follows from applying Definition \ref{var} $D-\tau$ times.

To show the second inequality we bound
\begin{eqnarray}
\left|\E\left[\prod_{t\in S} \left(Z^t-\E\left[Z^t\right]\right)^{d_t}\right]\right|
&\le&  \prod_{t\in S}  \left(\E\left[\left|Z^t-\E\left[Z^t\right]\right|^{d_t\cdot\frac{D}{td_t}}\right]\right)^{td_t/D} \nonumber\\
&\le&  \prod_{t\in S}  \left(\E\left[\left(|Z^t|+\E\left[|Z|^t\right]\right)^{\frac{D}{t}}\right]\right)^{td_t/D} \nonumber\\
&\le&  \prod_{t\in S}  \left(\E\left[2^{\frac{D}{t} - 1}\left(|Z|^{D}+(\E\left[|Z|^t\right])^{D/t}\right)\right]\right)^{td_t/D} \nonumber\\
&\le&  \prod_{t\in S}  \left(2^{\frac{D}{t}}\E\left[|Z|^{D}\right]\right)^{td_t/D} \nonumber\\
&=&2^{d}\E\left[|Z|^D\right] \nonumber \\
& \le & \frac{2^d L^{D-\tau} D!}{\tau!}\E\left[\left|Z\right|^{\tau}\right] \nonumber
\end{eqnarray}
where the first inequality uses H\"older's Inequality (see Lemma \ref{holder}),
the third inequality uses the fact (which follows from convexity) that $(x+y)^p \le 2^{p-1}(x^p + y^p)$ for any $p \ge 1$ and $x,y \ge 0$ (in particular ($x=|Z^t|$ and $y=\E\left[|Z|^t\right]$),
the fourth inequality uses Jensen's inequality,
and the last inequality uses the  inequality (\ref{eqn:applyMomBound}).
\end{proof}

We now prove the following intuitive fact that was left unproven near the end of Section \ref{sec:eachGraph}.
\begin{lemma}\label{lem:ordering}
In the notation of Section \ref{sec:eachGraph} there exists a canonical  ordering $h^{(1)},\ldots,h^{(k)}$ of the hyperedges ${\cal H}(G')$ such that $\nu_0=c$.
\end{lemma}

\begin{proof}
Let ${\cal L}$ be the line graph of $G'$, i.e.\ an undirected graph with one vertex for each of the $k$ hyperedges of $G'$ and an edge connecting every pair of vertices that correspond to hyperedges with intersecting vertex sets. We define the desired sequence of hyperedges $h^{(1)},\dots,h^{(k)}$ and a sequence of induced subgraphs ${\cal L}_1,\dots,{\cal L}_k$ of ${\cal L}$ as follows.

We set ${\cal L}_1$ to ${\cal L}$. For any $1 \le s \le k$ we form ${\cal L}_{s+1}$ from ${\cal L}_{s}$ by removing vertex $h^{(s)}$, where $h^{(s)}$ has the lowest label from the vertices of ${\cal L}_s$ subject to the constraint that the number of connected components $n_{s+1}$ of ${\cal L}_{s+1}$ must not exceed the number of connected components $n_s$ of ${\cal L}_s$. For example pick $h^{(s)}$ to be an arbitrary leaf of a depth first search tree started from an arbitrary vertex of ${\cal L}_s$ or an isolated vertex of ${\cal L}_s$ if there are any.

It remains to show that the ordering $h^{(1)},\dots,h^{(k)}$ satisfies the desired property $\nu_0=c$. Note that $h^{(s)}$ contributes to $\nu_0$ if and only if $V_s={\cal V}(h^{(s)})$, that is if and only if $h^{(s)}$ is an isolated vertex in ${\cal L}_s$. Whenever such an $h^{(s)}$ is chosen the number of connected components decreases by one (i.e.\ $n_{s+1}=n_s-1$), and otherwise the number of connected components is unchanged (i.e.\ $n_{s+1}=n_s$). We conclude that $\nu_0=n_{1}-n_{k+1}=c-0$ as desired.
\end{proof}

The following Lemma was used in the proof of the Initial Moment Lemma and will be used later in the proof of the Main Counting Lemma.

\begin{lemma}\label{lem:ziq1}
Let $G'$ be a labeled hypergraph with all degrees at least two, $c$ connected components with sets of vertices $C_1,\dots, C_c$ and the number of hyperedges $k_1,\dots, k_c$. Further, let $h^{(1)},\dots,h^{(k)}$
be the canonical ordering of its hyperedges specified in Lemma \ref{lem:ordering} where $k=\sum_{i=1}^ck_i$. Then for each $i=1,\dots,c$ we have
\[
(q-1)k_i - \sum_{v \in C_i} \delta_v \ge q-2
\]
where $\delta_v$ is the last power of vertex $v$ corresponding to the canonical ordering of hyperedges as defined in Section \ref{sec:eachGraph}.
\end{lemma}
\begin{proof}
Fix a labeled hypergraph $G'$. Recall that the power  of a vertex $v$ in hyperedge $h$ is denoted $\tau_{hv}$.
Following Section \ref{sec:eachGraph} define $V_j$ to be the set of vertices whose last incident hyperedge is $h^{(j)}$, where ``last'' is relative to the canonical ordering $h^{(1)},\dots,h^{(k)}$ of the hyperedges defined in that section. Let $I_i=\{s\in [k]|   {\cal V}(h^{(s)})\subseteq C_i\}$.   We charge $(q-1)k_i - \sum_{v \in C_i} \delta_v $ to the various hyperedges in the following natural way:
\begin{align}
(q-1)k_i - \sum_{v \in C_i} \delta_v & = \sum_{j\in I_i} \left((q-1) - \sum_{v \in V_j} \tau_{h^{(j)}v} \right) \equiv \sum_{j\in I_i} \alpha_j \label{eqn:ziq1}
.\end{align}
The contribution $\alpha_s$ of the   hyperedge with smallest index $s\in I_i$  is exactly $q-1$ since the degree of each vertex $v$ is at least two and $h^{(s)}$ is the first hyperedge in this connected component that we delete, i.e. $V_s=\emptyset$.
The last hyperedge $h^{(s')}$ for $s'=\max\{ j|j\in I_i\}$ clearly contributes $\alpha_{s'}=(q-1)-q=-1$ since $|V_{s'}|=\eta$. For any $j\in I_i\setminus \{s'\}$ we know that $\sum_{v \in V_j} \tau_{h^{(j)}v} \le q-1$ because $|V_j|\le \eta-1$. Otherwise component $i$ would contribute more than 1 to $\nu_0$ and $\nu_0$ would exceed $c$. We conclude that $\alpha_j \ge 0$ for $j\in I_i\setminus \{s'\}$. Using (\ref{eqn:ziq1}) and these lower bounds on the $\alpha_j$ we bound
\[
(q-1)k_i - \sum_{v \in C_i} \delta_v = \sum_{j\in I_i} \alpha_j \ge q-1 + 0 - 1 = q-2
\]
as desired.
\end{proof}
The following Corollary immediately follows from Lemma \ref{lem:ziq1}.
\begin{corollary}\label{Cor_Ineq}
Let $G'$ be a labeled hypergraph with all degrees at least two, $c$ connected components and the canonical ordering satisfying the condition $\nu_0=c$ we have $qk-\Delta\ge (q-2)\nu_0$ where
$\Delta=\sum_{v\in {\cal V}(G)} \delta_v=\sum_{t=0}^q (q-t)\nu_t$.
\end{corollary}
\section{General Even Moment Lemma}\label{ProofMainGen}
\begin{lemma}[H\"older's Inequality]\label{holder}
Let $p_1,\dots,p_k\in (1,+\infty)$ such that $\sum_{i=1}^k\frac{1}{p_i}=1$ then for arbitrary collection $X_1,\dots,X_k$ of random variables on the same probability space the following inequality holds
$$\E\left[\left|\prod_{i=1}^kX_i\right|\right]\le \prod_{i=1}^k\E\left[\left|X_i\right|^{p_i}\right]^{1/p_i}.$$
\end{lemma}
We will use the following corollary of H\"older's inequality known as Minkowski inequality (or triangle inequality for norms).
\begin{corollary}[Minkowski Inequality]\label{cor:momentOfSum}
Let $k\ge 1$  and $Z_1,Z_2,\ldots,Z_m$ be (potentially dependent) random variables with $\E[|Z_i|^k] \le z_i^k$ for $z_i\in R_+$. It follows that
\begin{align}
\E\left[ \left(\sum_{i=1}^m|Z_i|\right)^k\right] & \le \left(\sum_{i=1}^mz_i\right)^k \label{eqn:momentOfSum}
.\end{align}
\end{corollary}

\begin{lemma}[General Even Moment Lemma]\label{generalmoment}
We are given $n$ independent moment bounded random variables $Y=(Y_1,\dots, Y_n)$ with the same parameter $L$ and a general power $q$ polynomial $f(x)$ and maximal variable power $\Gamma=\max_{h\in {\cal  H}, v\in h} \tau_{hv}$. Let $k\ge 2$ be an even integer then
 \begin{eqnarray}\label{inequality4gen}
\E\left[\left| f(Y)-\E\left[ f(Y)\right]\right|^k\right] & \le & \max \left\{ \max_{t\in [q]}\left(\sqrt{kR_4^q\Gamma^tL^t \mu_t\mu_0}\right)^{k}, \max_{t\in [q]}(k^{t}R_4^q\Gamma^tL^t\mu_t)^{k}\right\}.
\end{eqnarray}
where $R_4\ge 1$ is some absolute constant.
\end{lemma}

\begin{proof}
Let weight function $w$ and hypergraph $H=([n],{\cal H})$ be such that $f(Y) = \sum_{h \in {\cal H}} w_h \prod_{v \in \ve{h}} Y^{\tau_{hv}}$. Let $X_{hv} = Y_v^{\tau_{hv}}-\Eb{Y_v^{\tau_{hv}}}$.
Let ${\cal H}'$ denote the set of all possible hyperedges (including the empty hyperedge) with vertices from ${\cal V}(H)=[n]$ and total power at most $q$.
First we note that
\begin{align}
f(Y) &= \sum_{h\in {\cal H}} w_h \prod_{v \in \ve{h}} (X_{hv} + \Eb{Y_v^{\tau_{hv}}}) \nonumber \\
& = \sum_{h' \in {\cal H}'} \sum_{h \in {\cal H} : h \extends h'} w_h \left(\prod_{v \in \ve{h} \setminus \ve{h'}} \Eb{Y_v^{\tau_{hv}}}\right)\left( \prod_{v \in \ve{h'}} X_{hv}\right) \nonumber \\
& = \sum_{h' \in {\cal H}'} w'_{h'} \prod_{v \in \ve{h'}} X_{h'v} \label{eqn:center}
\end{align}
where $h'$ ranges over all possible hyperedges (including the empty hyperedge) and
\[
w'_{h'} = \sum_{h\in {\cal H}|\ h \extends h'}w_h\left(\prod_{v \in \ve{h} \setminus \ve{h'}} \Eb{Y_v^{\tau_{hv}}}\right)
.\]
We next group the monomials on the right hand side of (\ref{eqn:center}) by cardinality, power, and sign of coefficient, yielding $m \le 2q^2$ polynomials $g^{(1)},\dots,g^{(m)}$ with corresponding weight functions for all monomials $w^{(1)},\dots,w^{(m)}$ and powers $q_1,\dots,q_m$. That is,
\begin{align}
f(Y) & = w'_{\{\}} + \sum_{i=1}^m \sum_{h':\eta(h') \ge 1} w^{(i)}_{h'} \prod_{v \in \ve{h'}} X_{h'v} \\
& = \Eb{f(Y)} + \sum_{i=1}^m g^{(i)}(Y)
\nonumber
\end{align}
where $\{\}$ is the empty hyperedge.
We have
\begin{align*}
\mu_r(w^{(i)}, Y) \le \mu_r(w', Y) &= \max_{h_0:q(h_0)=r} \sum_{h' \extends h_0} |w'_{h'}| \prod_{v \in \ve{h'} \setminus \ve{h_0}} \Eb{|Y^{\tau_{h'v}}_v|} \\
 & \le \max_{h_0:q(h_0)=r} \sum_{h' \extends h_0} \sum_{h \extends h'}|w_h|\left(\prod_{v \in \ve{h} \setminus \ve{h'}} |\Eb{Y^{\tau_{hv}}_v}|\right) \prod_{v \in \ve{h'} \setminus \ve{h_0}} \Eb{|Y^{\tau_{hv}}_v|} \\
 & \le \max_{h_0:q(h_0)=r} \sum_{h' \extends h_0} \sum_{h \extends h'}|w_h|\left(\prod_{v \in \ve{h} \setminus \ve{h_0}} \Eb{|Y_v|^{\tau_{hv}}}\right) \\
 & \le 2^q \max_{h_0:q(h_0)=r} \sum_{h \extends h_0} |w_h|\left(\prod_{v \in \ve{h} \setminus \ve{h_0}} \Eb{|Y_v|^{\tau_{hv}}}\right) = 2^q\mu_r(w,Y)=2^q\mu_r
\end{align*}
where we upper bounded the number of hyperedges $h'$ such that $h \extends h' \extends h_0$ by $2^q$.
Therefore, for even $k\ge 2$ the Lemma \ref{lem:moment} implies that
$$\E\left[\left|g^{(i)}(Y)\right|^k\right]=\left|\E\left[g^{(i)}(Y)^k\right]\right|\le 2^{qk} \max \left\{ \left(\sqrt{kR_3^{q_i}\Gamma^{q_i}L^{q_i} \mu_{q_i}\mu_0}\right)^{k}, \max_{t\in [q_i]}(k^{t}R_3^{q_i} L^t\Gamma^t\mu_t)^{k}\right\}=2^{qk}z_i^k.$$

Applying Corollary \ref{cor:momentOfSum} yields
\begin{eqnarray*}
\E\left[\left|f(Y)-\E\left[ f(Y)\right]\right|^k\right]&\le& \E\left[\left(\sum_{i=1}^{m}\left|g_i(Y)\right|\right)^k\right]\le  2^{qk} \left(\sum_{i=1}^{m}z_i\right)^k\le 2^{qk} m^k\max_{i\in [m]}z^k_i\\
&\le& 2^{qk} m^k\cdot \max \left\{ \max_{t\in [q]}\left(\sqrt{kR_3^{q}\Gamma^tL^{t} \mu_{t}\mu_0}\right)^{k}, \max_{t\in [q]}(k^{t}R_3^qL^t\Gamma^t\mu_t)^{k}\right\}\\
&\le&  \max \left\{ \max_{t\in [q]}\left(\sqrt{kR_4^{q}\Gamma^tL^{t} \mu_{t}\mu_0}\right)^{k}, \max_{t\in [q]}(k^{t}R_4^qL^t\Gamma^t\mu_t)^{k}\right\}
\end{eqnarray*}
for some absolute constant $R_4$ such that $m^22^{2q} R_3^q\le R_4^q$.
\end{proof}

\section{Proof of the Theorem \ref{main1}}\label{ProofMain}

Now we prove Theorem \ref{main1} by applying the Markov's inequality.
\begin{proof}
By Markov's inequality we derive
$$Pr[|f(Y)-\E\left[ f(Y)\right]|\ge \lambda]=Pr[|f(Y)-\E\left[ f(Y)\right]|^k\ge \lambda^k] \le \frac{\E[|f(Y)-\E\left[ f(Y)\right]|^k]}{\lambda^k}.$$
Choosing $k^*\ge 0$ to be the  even   integer  such that $k^*\in (K-2,K]$ for
$$K=\min \left\{ \min_{t\in [q]}\frac{\lambda^2}{e^2R_4^q\Gamma^tL^t\mu_t \mu_0 },\min_{t\in [q]}\left(\frac{\lambda}{ e R_4^qL^t\Gamma^t\mu_t}\right)^{1/t}\right\}$$
i.e.\
\[
\frac{\sqrt{k^*R_4^q\Gamma^tL^t \mu_t\mu_0}}{\lambda} \le 1/e \text{ and } \frac{(k^*)^{t} R_4^q L^t \Gamma^t \mu_t}{\lambda} \le 1/e
\]
for all $t\in [q]$. Using the inequality (\ref{inequality4gen}) from the Lemma \ref{generalmoment} we derive
\begin{eqnarray*}
Pr[|f(Y)-\E\left[ f(Y)\right]|\ge \lambda] &\le& \frac{\E\left[|f(Y)-\E\left[ f(Y)\right]|^{k^*}\right]}{\lambda^{k^*}} \\
& \le &
\max\left\{\max_{t\in [q]} e^{k^*\ln \frac{\sqrt{k^*R_4^q\Gamma^tL^t \mu_t\mu_0}}{\lambda}}, \max_{t\in [q]} e^{k^*\ln  \frac{(k^*)^{t} R_4^q L^t\Gamma^t \mu_t}{\lambda}   }\right\}\\
&\le& e^{-k^*}\le  e^{-K+2}\nonumber\\
&\le& e^2 \cdot \max\left\{\max_{t\in [q]}e^{-\frac{\lambda^2}{R^q \Gamma^tL^t\mu_t\mu_0}},\max_{t\in [q]}e^{-\frac{1}{L\cdot \Gamma}\left(\frac{\lambda}{R^q\mu_t}\right)^{1/t}}\right\}, \label{eqn:centeredMoment}
\end{eqnarray*}
for some universal constant $R>R_4\ge 1$. This implies the statement of the Theorem.
\end{proof}


\section{Counting Lemma}\label{sec:counting}
In this section we consider labeled hypergraphs in ${\cal S}_2(k,\ell)$ for fixed parameters $\eta, q$ and $\Gamma$.
We use $C_1,\ldots,C_c$ to denote the set of vertices in the connected components of a labeled hypergraph. We use $\ell_1,\ldots,\ell_c$ and $k_1,\ldots,k_c$ to denote the number of vertices and hyperedges in those connected components.
We will freely use the following elementary facts:
\begin{enumerate}
\item $\eta \le \ell_i \le \eta k_i / 2$ where the lower bound follows from the fact that each connected component has at least one hyperedge and the upper bound follows from the fact that each vertex has degree at least two;
\item $\eta c \le \ell \le \eta k/2$ (these inequalities are obtained by summing up the above inequalities over all connected components);
\item $1 \le c \le k/2$, the lower bound is obvious and the upper bound follows from the previous inequality.
\end{enumerate}

In two of the auxiliary lemmas below we will use the classical Gibbs inequality which states that for two arbitrary discrete probability distributions $p_1,\dots,p_n$ and $q_1,\dots,q_n$ with strictly positive $p_i,q_i$ the following inequality holds
$$-\sum_{i=1}^np_i\log_2p_i\le -\sum_{i=1}^np_i\log_2q_i$$
or equivalently
$$\prod_{i=1}^n p_i^{p_i} \ge \prod_{i=1}^n q_i^{p_{i}}.$$

In what follows we identify a vertex with its index $v\in [\ell]$. The main statement of this section is the following
\begin{lemma}[Main Counting Lemma]\label{MainCount}
For any $k$, $\Gamma$, $q\ge \eta \ge 1$, $\ell$, $c$,  $\bar{D}$, $\bar{d} \ge \bar{2}$ and $\bar{\delta}$ we have
\begin{align*}
|{\cal S}(k,\ell, c,  \bar{d}, \bar{D}, \bar{\delta})| \left(\prod_{v\in [\ell]} \frac{D_v!}{\delta_v!}\right) &\le   R_0^{qk} \Gamma^{qk - \ell - \sum_{v\in [\ell]} \delta_v} k^{qk - c(q-1) - \sum_{v\in [\ell]} (\delta_v - 1)},
\end{align*}
for some universal constant $R_0>1$.
\end{lemma}

We prove Lemma \ref{MainCount} as a sequence of auxiliary Lemmas.

We say that $C_1,\ldots,C_c$ and $k_1,\ldots,k_c$ are \emph{feasible} (with respect to $\bar{d}, \bar{D},\bar{\delta}$ clear from context) if there is a labeled hypergraph in ${\cal S}(k, \ell, c, \bar{d}, \bar{D},\bar{\delta})$ with corresponding canonical oredring of its hyperedges whose connected components (numbered arbitrarily) have vertex sets $C_1,\ldots,C_c$ and number of hyperedges $k_1,\ldots,k_c$.

\begin{lemma}\label{lem:S0v2}
For any $k$, $\ell$, $q$, $\eta$, and $\bar{d}$ we have
\[
|\cup_{c,\bar{D},\bar{\delta}} {\cal S}(k,\ell, c, \bar{d}, \bar{D}, \bar{\delta})| \le \binom{q-1}{\eta - 1}^k \prod_{v \in [\ell]} \binom{k}{d_v} \le 2^{(q-1)k} \frac{k^{\eta k}}{\prod_v d_v!}
.\]
Note, that we intentionally do not fix $\Gamma$ since the bound in the Lemma holds for any $\Gamma\le q$.
\end{lemma}
\begin{proof}
Fix $k$, $\ell$, $\eta$, 
$q$, and $\bar{d}$. To show the first inequality, note that a labeled hypergraph is uniquely specified by:
\begin{enumerate}
\item for every vertex $v=1,\dots,\ell$ whether or not it appears in each of the $k$ hyperedges and
\item for every hyperedge $h$ the power vector of its $\eta$ vertices.
\end{enumerate}
Vertex $v$ with degree $d_v$ clearly has $\binom{k}{d_v}$ possible sets of hyperedges it can appear in, so there are at most $\prod_{v \in [\ell]} \binom{k}{d_v}$ ways to assign vertices to the hyperedges.
In general this is quite a rough estimate (since we do not use the fact that each hyperedge contains exactly $\eta$ vertices). More precisely we would like to estimate the number of $0/1$ matrices of dimension $k\times \ell$ with prescribed row sums equal to $\eta$ and a column sum $d_v$ for a row indexed by $v$. Estimating this quantity is an important topic in combinatorics (see survey \cite{S1} and references therein) but for our purposes the simple estimate above provides a tight bound.

We now count the ways to assign weights to the hyperedges. Recall that $\tau_{hv}$ denote the weight of the $v^{\text{th}}$ vertex in hyperedge $h$. There is a standard bijection between $q-1$ digit binary strings with $\eta-1$ zeros and placements of $q$ identical items into $\eta$ bins such that each bin has at least one item. The string starts with $\tau_{1,h}-1$ ones followed by a zero, followed by $\tau_{2,h}-1$ ones followed by a zero and so on, ending with $\tau_{q,h}-1$ ones (and no trailing zero). We conclude that there are $\binom{q-1}{\eta - 1}^k$ ways to assign weights of the hyperedges. This concludes the proof of the first inequality.

The second inequality follows because $\binom{q-1}{\eta - 1} \le 2^{q-1}$, $\binom{k}{d_v} \le k^{d_v} / d_v!$, and $\sum_v d_v = \eta k$.
\end{proof}

\begin{lemma}\label{lem:S1}
For any $k$, $\ell$, $c$, $q$, $\eta$, $\bar{d} \ge \bar{2}$, $\bar{D}$ and $\bar{\delta}$ we have
\[
|{\cal S}(k,\ell, c, \bar{d}, \bar{D}, \bar{\delta})| \le 2^{\ell + k + 5c + qk} k! \left(\prod_{v\in[\ell]} \frac{1}{d_v!} \right)
\max_{\substack{C_1,\ldots,C_c \\ k_1,\ldots,k_c \\ \text{feasible}}}  \left(\prod_{i=1}^c \frac{k_i^{\eta k_i}}{k_i!} \binom{\ell}{|C_i| - 1} \right)
\]
where the maximums are evaluated over all $C_1,\ldots,C_c$ and $k_1,\ldots,k_c$ that are feasible as defined above. This bound holds for any $\Gamma\le q$.
\end{lemma}

\begin{proof}
We prove the Lemma by mapping the labeled hypergraphs in ${\cal S}(k, \ell, c, \bar{d}, \bar{D}, \bar{\delta})$ into distinct binary strings and bounding the length of these strings.

Fix an arbitrary hypergraph in ${\cal S}(k, \ell, c, \bar{d}, \bar{D}, \bar{\delta})$. Our encoding begins by encoding the vertices in the connected component that contains vertex 1. Let the vertices in this component be denoted by $C_1$ and $|C_1| = \ell_1$. We encode $\ell_1$ in unary, e.g.\ our string begins with 1110 if $\ell_1=3$. We then encode the identity of the remaining $\ell_1 - 1$ vertices in $C_1 \setminus \set{1}$ using a single character with $\ceil{\log_2 \binom{\ell}{\ell_1 - 1}}$ binary digits, i.e.\ we have $\binom{\ell}{\ell_1 - 1}$ options which are encoded in binary.

We then look at the lowest-indexed vertex that has yet to be placed in a connected component and encode the size $\ell_2$ and vertices $C_2$ of its component in the same manner. We repeat until all $\ell$ vertices have been placed in one of the $c$ connected components, where the $i^{\text{th}}$ component considered has $\ell_i$ vertices. At this point we have partitioned the vertices into connected components using
\begin{equation}
\sum_{i=1}^c \left[\ell_i + 1 + \ceil{\log_2 \binom{\ell}{\ell_i - 1}}\right] = \ell + c + \sum_{i=1}^c \ceil{\log_2 \binom{\ell}{\ell_i - 1}} \label{eqn:CCVertices}
\end{equation}
 bits.

We then encode the number of hyperedges $k_i$ in each connected component in unary using
\begin{equation}
 \sum_{i=1}^c (k_i + 1) = k + c \label{eqn:CCNumEdges}
\end{equation}
 bits, i.e.\ we have $k_1$ ones followed by a zero, followed by $k_2$ ones followed by a zero and so on. There are $\frac{k!}{k_1!\cdot\ldots\cdot k_c!}$ ways to partition  $k$ indices into $c$ groups where $i$-th group has $k_i$ indices. Therefore, we can encode which component each hyperedge is in using
\begin{equation}
 \ceil{\log_2 \left(\frac{k!}{k_1!\cdot\ldots\cdot k_c!}\right)} \label{eqn:CCEdges}
\end{equation}
bits.

Finally we encode the vertices and the power vectors of each hyperedge in component $1\le i \le c$. The number of possibilities is clearly $|{\cal S}(k_i, \ell_i, 1, \bar{d}|_i,\bar{D}|_i,\bar{\delta}|_i)|$, where $\bar{d}|_i$ (resp.\ $\bar{D}|_i$ and $\bar{\delta}|_i$) is the vector of degrees (resp.\ total powers and final powers) of the vertices in component $i$. Using Lemma \ref{lem:S0v2} we bound $|{\cal S}(k_i, \ell_i, 1, \bar{d}|_i,\bar{D}|_i,\bar{\delta}|_i)| \le \frac{2^{qk_i} k_i^{\eta k_i}}{\prod_{v \in C_i} d_v!}$. Therefore, the total number of bits used to encode the hyperedges is at most
\begin{equation}
\ceil{\log_2 \left(\left(\prod_v \frac{1}{d_v!}\right) 2^{qk} \left(\prod_{i=1}^{c} k_i^{\eta k_i} \right)\right)} \label{eqn:edges}
\end{equation}
bits.

Combining, $(\ref{eqn:CCVertices}), (\ref{eqn:CCNumEdges}), (\ref{eqn:CCEdges}), (\ref{eqn:edges})$ we obtain that the total number of bits used to encode an arbitrary hypergraph in ${\cal S}(k, \ell, c, \bar{d}, \bar{D}, \bar{\delta})$ is upper bounded by the maximum over feasible $C_1,\ldots,C_c$ and $k_1,\ldots,k_c$ of
\begin{multline*}
\ell + c
+ \sum_{i=1}^c \ceil{\log_2 \binom{\ell}{\ell_i - 1}}
+ k + c +
 \ceil{\log_2 \left(\frac{k!}{k_1!\cdot\ldots\cdot k_c!}\right)} +
 \ceil{\log_2 \left(\left(\prod_v \frac{1}{d_v!}\right) 2^{qk} \left(\prod_{i=1}^{c} k_i^{\eta k_i} \right)\right)}
,\end{multline*}
which we'll denote by $b$. We can safely add trailing zeros so that each hypergraph is encoded using exactly $b$ bits. We conclude that the number of hypergraphs is at most $2^b$, which is at most the right-hand side of the Lemma statement, as desired.
\end{proof}

\begin{lemma}\label{lem:S2}
For any $k$, $\Gamma$, $\ell$, $c$, $q \ge \eta \ge 1$, $\bar{D}$, $\bar{d} \ge \bar{2}$ and $\bar{\delta}$ we have
\begin{align}
\lefteqn{\left(\prod_v \frac{D_v!}{\delta_v!}\right) |{\cal S}(k,\ell, c,  \bar{d}, \bar{D},\bar{\delta})|} \nonumber \\
 & \le  e^{O(qk)} k^{qk - \sum_{v \in [\ell]} (\delta_v - 1)} \Gamma^{qk - \ell - \sum_{v \in [\ell]} \delta_v}\max_{\substack{C_1,\ldots,C_c \\ k_1,\ldots,k_c \\ \text{feasible}}} \prod_{i=1}^c \left(\frac{k_i}{k}\right)^{(q-1) k_i - \sum_{v \in C_i} (\delta_v - 1)-(|C_i|-1)}
 \label{eqn:S2.2}
.\end{align}
\end{lemma}


\begin{proof}
Applying Lemma \ref{lem:S1} we get
\begin{align}
\lefteqn{\left(\prod_v \frac{D_v!}{\delta_v!}\right)|{\cal S}(k,\ell, c, \bar{d}, \bar{D}, \bar{\delta})|} \nonumber \\
 & \le \max_{\substack{C_1, \ldots, C_c \\ k_1,\ldots,k_c}} \left[
\Bigg(2^{\ell + k + 5c + q k} k!\Bigg) 
 \cdot  \left(\prod_{v=1}^\ell \frac{D_v!}{d_v!\delta_v!}\right) 
  \cdot \left(\prod_{i=1}^c \binom{\ell}{|C_i| - 1} \right) 
   \cdot \left(\prod_{i=1}^c \frac{k_i^{\eta k_i}}{k_i!} \right) 
\right] \label{1234}
.\end{align}
We will now bound each factor of (\ref{1234}) in turn, making frequent use of the formula $n! = (n/e)^n n^{-O(1)}$ and the inequality $\binom{n}{m} \le n^m / m!$.

First we bound
\begin{align}
2^{\ell + k + 5c + qk} k! &\le e^{O(qk)} k^k\label{eqn:stirling1}
.\end{align}

Secondly we bound
\begin{align}
\prod_{v=1}^\ell \frac{D_v!}{d_v!\delta_v!} 
& = \prod_{i=1}^c \prod_{v \in C_i}  \frac{(D_v-d_v)!}{\delta_v!}\cdot \binom{D_v}{d_v} \nonumber \\
& \le \prod_{i=1}^c \prod_{v \in C_i} D_v^{\max(0,D_v - d_v - \delta_v)} 2^{D_v} \nonumber \\
& \le \prod_{i=1}^c \prod_{v \in C_i} D_v^{D_v - d_v - \delta_v + 1} 2^{D_v} \nonumber \\
& \le \prod_{i=1}^c (\Gamma k_i)^{qk_i - \eta k_i - \sum_{v \in C_i} (\delta_v - 1)} 2^{q k_i} \nonumber \\
& = \Gamma^{(q-\eta)k - \sum_{v \in [\ell]} (\delta_v - 1)} 2^{qk} \prod_{i=1}^c k_i^{(q-\eta)k_i - \sum_{v \in C_i} (\delta_v - 1)}  \nonumber
\end{align}
using the facts $\sum_{v \in C_i} D_v = qk_i$, $\sum_{v \in C_i} d_v = \eta k_i$ and $D_v \le \Gamma k_i \le q k_i$.
We finally observe that $(q-\eta)k - \sum_{v \in [\ell]} (\delta_v - 1) \le qk - 2 \ell  - \sum_{v \in [\ell]} \delta_v + \ell$, yielding
\begin{align}
\prod_{v=1}^\ell \frac{D_v!}{d_v! \delta_v!} 
& \le \Gamma^{qk - \ell - \sum_{v \in [\ell]} \delta_v} 2^{qk} \prod_{i=1}^c k_i^{(q-\eta)k_i - \sum_{v \in C_i} (\delta_v - 1)}
 \label{eqn:stirling2}
.\end{align}

For the third factor we consider two cases. If $\eta \ge 2$ we have $\ell_i \ge 2$ and $c \le \ell / 2$ and we bound
\begin{align}
\prod_{i=1}^c \binom{\ell}{\ell_i - 1} &\le \prod_{i=1}^c \frac{\ell^{\ell_i - 1}}{(\ell_i - 1)!} \le \prod_{i=1}^c \frac{e^{O(\ell_i)} \ell^{\ell_i - 1}}{(\ell_i-1)^{\ell_i-1}} \nonumber \\
&= \prod_{i=1}^c \frac{e^{O(\ell_i)} (\ell-c)^{\ell_i - 1}}{(\ell_i-1)^{\ell_i-1}}= e^{O(\eta k)} \prod_{i=1}^c \left(\frac{\ell_i-1}{\ell-c}\right)^{-(\ell_i-1)} \nonumber \\
&\le e^{O(\eta k)} \prod_{i=1}^c \left(\frac{k_i}{k}\right)^{-(\ell_i-1)}
\label{eqn:stirling3}
\end{align}
where the last inequality is Gibbs'. It turns out that (\ref{eqn:stirling3}) holds when $\eta=1$ as well because every $\ell_i=1$ and hence both $\prod_{i=1}^c \binom{\ell}{\ell_i - 1}$ and $\prod_{i=1}^c \left(\frac{k_i}{k}\right)^{-(\ell_i-1)}$ are equal to 1.

Fourth we write
\begin{align} 
\prod_{i=1}^c \frac{k_i^{\eta k_i}}{k_i!}= e^{O(\eta k)} \prod_{i=1}^c k_i^{(\eta-1)k_i}. \label{eqn:stirling4}
\end{align}

Combining (\ref{eqn:stirling1}), (\ref{eqn:stirling2}), (\ref{eqn:stirling3}), (\ref{eqn:stirling4}) and (\ref{1234}) we get
\begin{align*}
\lefteqn{\left(\prod_v \frac{D_v!}{\delta_v!}\right)|{\cal S}(\ell, c,  \bar{d}, \bar{D}, \bar{\delta})|} \nonumber \\
 & \le \max_{\substack{C_1,\ldots,C_c \\ k_1,\ldots,k_c}}
 \Bigg[
\Bigg(e^{O(qk)} k^k\Bigg)
 \cdot  \left(\Gamma^{qk - \ell - \sum_{v \in [\ell]} \delta_v} 2^{qk} \prod_{i=1}^c k_i^{(q-\eta)k_i - \sum_{v \in C_i} (\delta_v - 1)}\right) \cdot  \\
  &~~~~~~~~~~~~~~\cdot \left(e^{O(\eta k)} \prod_{i=1}^c \left(\frac{k_i}{k}\right)^{-(\ell_i-1)} \right)
   \cdot \left(e^{O(\eta k)} \prod_{i=1}^c k_i^{(\eta-1)k_i}\right) \Bigg] \\
& = e^{O(qk)} k^{k} \Gamma^{qk - \ell - \sum_{v \in [\ell]} \delta_v}\max_{\substack{C_1,\ldots,C_c \\ k_1,\ldots,k_c}} \prod_{i=1}^c k_i^{(q-1) k_i - \sum_{v \in C_i} (\delta_v - 1)}\left(\frac{k_i}{k}\right)^{-(\ell_i-1)} \\
& = e^{O(qk)} k^{qk - \sum_{v \in [\ell]} (\delta_v - 1)} \Gamma^{qk - \ell - \sum_{v \in [\ell]} \delta_v}\max_{\substack{C_1,\ldots,C_c \\ k_1,\ldots,k_c}} \prod_{i=1}^c \left(\frac{k_i}{k}\right)^{(q-1) k_i - \sum_{v \in C_i} (\delta_v - 1)-(\ell_i-1)}
.\end{align*}
\end{proof}

Our final counting lemma bounds the optimization problem of Lemma \ref{lem:S2}.

\begin{lemma}\label{count}
For any $k$, $\ell$, $c$, $q \ge \eta \ge 1$, $\bar{d} \ge \bar{2}$, $\bar{D}$, and $\bar{\delta}$ we have
\begin{align*}
\max_{\substack{C_1,\ldots,C_c \\ k_1,\ldots,k_c \\ \text{feasible}}} \prod_{i=1}^c \left(\frac{k_i}{k}\right)^{(q-1) k_i - \sum_{v \in C_i} (\delta_v - 1)-(|C_i|-1)} &\le k^{-(c-1)(q-1)} e^{O(qk)}
.\end{align*}
\end{lemma}

\begin{proof}
We are looking to upper-bound
\begin{align}
\max_{\substack{\alpha_1,\ldots,\alpha_c \\ z_1,\ldots,z_c \\ \text{feasible}}} \prod_{i=1}^c \alpha_i^{z_i} \equiv {\cal M} \label{eqn:r}
\end{align}
where  $z_i = (q-1)k_i - (|C_i|-1) - \sum_{v \in C_i} (\delta_v - 1) = 1 + (q-1)k_i - \sum_{v \in C_i} \delta_v$ and $\alpha_i = k_i / k$.

We upper-bound ${\cal M}$ by the relaxation
\begin{align}
\max_{\substack{\alpha_1,\ldots,\alpha_c \\ z_1,\ldots,z_c}} & \prod_{i=1}^c \alpha_i^{z_i} \text{ such that} \label{eqn:obj} \\
\sum_i \alpha_i &= 1 \label{eqn:c1} \\
\alpha_i &\ge 0\label{eqn:c2} \\
\sum_i z_i &= Z \label{eqn:c3} \\
z_i &\ge q - 1\label{eqn:c4}
\end{align}
where $Z = \sum_i z_i = c + (q-1)k - \sum_v \delta_v$. To show this is a relaxation we need to prove that any $z_i$ feasible in (\ref{eqn:r}) satisfies (\ref{eqn:c4}), which follows from Lemma \ref{lem:ziq1} which states that
$(q-1)k_i - \sum_{v \in C_i} \delta_v\ge q-2$. Another implication of that lemma is that $Z\ge (q-1)c$.

When $q=1$ we can trivially prove the Lemma by upper-bounding (\ref{eqn:obj}) by 1, so we hereafter assume $q \ge 2$ and hence every $z_i$ is strictly positive. For any fixed $\set{z_i}_{i}$, Gibbs' inequality implies that the maximum of (\ref{eqn:obj}) occurs when $\alpha_i = z_i / Z$. Therefore we have reduced our problem to
\begin{align}
\max_{z_1,\ldots,z_c} & \prod_{i=1}^c \left(\frac{z_i}{Z}\right)^{z_i} \text{ such that} \label{eqn:obj2}\\
\sum_i z_i &= Z \nonumber \\
z_i &\ge q - 1\nonumber
.\end{align}
Clearly the optimum is when $z_i=q-1$ for all $i\ne 1$ and $z_1 = Z - (c-1)(q-1)$. Therefore the maximum of (\ref{eqn:obj2}) is
\begin{align}
\left(\frac{Z-(c-1)(q-1)}{Z}\right)^{Z-(c-1)(q-1)} \left(\frac{q-1}{Z}\right)^{(c-1)(q-1)} & \le 1 \cdot \left(\frac{1}{c}\right)^{(c-1)(q-1)} \nonumber \\
 &= \left( \frac{1}{k}\right)^{(c-1)(q-1)} \left(\frac{k}{c}\right)^{(c-1)(q-1)} \nonumber \\
 & \le k^{-(c-1)(q-1)} e^{O(qk)} \label{eqn:gensym}
\end{align}
using the fact that $Z = \sum_i z_i \ge (q-1)c$ in the first inequality.
\end{proof}

We are finally ready to prove our Main Counting Lemma.
\begin{proof} {\bf of Lemma \ref{MainCount}}.
Lemmas \ref{lem:S2} and \ref{count} give us
\begin{align*}
\lefteqn{\left(\prod_v \frac{D_v!}{\delta_v!}\right) |{\cal S}(k,\ell, c,  \bar{d}, \bar{D}, \bar{\delta})|} \\
 & \le  e^{O(qk)} k^{qk - \sum_{v \in [\ell]} (\delta_v - 1)} \Gamma^{qk - \ell - \sum_{v \in [\ell]} \delta_v} k^{-c(q-1) + (q-1)} \\
 & = e^{O(qk)} \Gamma^{qk - \ell - \sum_{v \in [\ell]} \delta_v} k^{qk - c(q-1) - \sum_{v \in [\ell]} (\delta_v - 1)} \\
 & \le R_0^{qk} \Gamma^{qk - \ell - \sum_{v \in [\ell]} \delta_v} k^{qk - c(q-1) - \sum_{v \in [\ell]} (\delta_v - 1)}
\end{align*}
for some absolute constant $R_0>1$ ( we used the fact that $k^{q-1}=e^{O(qk)}$).
\end{proof}


\section{Permanents of Random Matrices}\label{sec:perm}
\begin{proof}{\bf of Theorem \ref{perm}}
Notice first that $\mu_t\le (n-t)!\le n^{n-t}$ and the power of the polynomial $q=n$. Since the permanent is a multilinear polynomial and $\Eb{Y_{ij}}=0$ we can directly apply  the Lemma \ref{lem:momentPrelim} for $k\le n$. Note also that $n$ in this Theorem is the dimension of the matrix and not the number of random variables as in Lemma \ref{lem:momentPrelim} (which is $n^2$ in this setting). We obtain
 \begin{eqnarray*}
 \left|\E\left[P(A)^k\right]\right|  &\le&  \max_{\ell,{\bar \nu}}\left\{R^{nk}k^{nk-\ell}\prod_{t=0}^nn^{(n-t)\nu_t}\right\}\\
 &=&  \max_{\ell}\left\{R^{nk}k^{nk-\ell}n^{\ell}\right\}=R^{nk}k^{nk}\max_{\ell}\left\{\left(\frac{n}{k}\right)^{\ell}\right\}\le R^{nk}k^{nk/2}n^{nk/2}.
 \end{eqnarray*}

We fix  the deviation $\lambda=t\sqrt{n!}>0$ and choose $k^*$ to be the even number in the interval $(K-2,K]$ for
$K= \frac{( \lambda/e)^{2/n}}{R^2n}.$
Using the Markov's inequality we derive
\begin{eqnarray*}
Pr[|P(A)|\ge \lambda] &\le& \frac{\E\left[|P(A)|^{k^*}\right]}{\lambda^{k^*}} \\
& \le & e^{k^*\ln \frac{R^n(k^*)^{n/2}n^{n/2}}{\lambda} }  \le e^{-k^*}\le  e^{-K+2}\nonumber\\
&\le& e^2 \cdot  e^{-\frac{( \lambda/e)^{2/n}}{R^2n}}\le e^2 \cdot  e^{-c\dot t^{2/n}}, \label{eqn:centeredMoment1}
\end{eqnarray*}
for some absolute constant $c>0$. Note that condition $k^*\le n$ is implied by the condition $K\le n$ which in turn is equivalent to the condition $\lambda\le e R^{n}n^n$. If $\lambda>eR^{n}n^n$ then we choose $k^*=n$ and  estimate
\begin{eqnarray*}
Pr[|P(A)|\ge \lambda]&<&Pr[|P(A)|\ge eR^{n}n^n]\\
&\le&  \frac{\E\left[|P(A)|^{k^*}\right]}{(e R^{n}n^n)^{k^*}}\le e^{-n}
\end{eqnarray*}
\end{proof}
\begin{proof}{\bf of Theorem \ref{perm1}}
We have an $n$ by $n$ matrix $A$ with entries that are independent except that the matrix is symmetric.
The permanent $P(A)$ is a degree $n$ polynomial of independent random variables with maximal variable degree $\Gamma=2$. (Note that the number of variables is $\binom{n}{2}+n$, not $n$.)
The permanent is a sum of products over permutations. We also treat each such permutation $\pi$ as a set of pairs $(i,j)$ for row index $i$ and column index $j$.
We write $h(\pi)$ for the hyperedge $h$ corresponding to $\pi$. More generally for any set $S$ of matrix entries we write $h(S)$ for the corresponding hyperedge.
Note that because each variable appears in up to two positions in the matrix the mapping $h$ is not a bijection.
Clearly
\begin{align*}
P(A) &= \sum_{\pi} \prod_{i=1}^n A_{i,\pi(i)} = \sum_h \sum_{\pi: h(\pi) = h} \prod_{i=1}^n A_{i,\pi(i)} \\
&= \sum_h \underbrace{\left(\sum_{\pi: h(\pi) = h} 1\right)}_{\equiv w_h} \left(\prod_{v \in \ve{h}} Y_v^{\tau_{hv}}\right)
\end{align*}

As in the proof of Lemma \ref{generalmoment} we write $P(A)$ as the sum of polynomials $g^{(1)},\dots,g^{(m)}$ with weights $w^{(1)},\dots,w^{(m)}$ and total powers  $q_1,\dots,q_m$ (in this case $\Eb{P(A)}=0$).

Fix some $1 \le i \le m$ and hyperedge $h'$ with $q(h')=q_i$. The next step is to bound coefficients of polynomials $g^{(i)}$,
\begin{align*}
w^{(i)}_{h'} &= \sum_{h \extends h'} w_h \left(\prod_{v \in \ve{h} \setminus \ve{h'}} \Eb{Y_v^{\tau_{hv}}}\right) \\
&= \sum_{\pi : h(\pi) \extends h'} \left(\prod_{v \in \ve{h(\pi)} \setminus \ve{h'}} \Eb{Y_v^{\tau_{hv}}}\right) \\
& \le \sum_{S: h(S)=h'} \sum_{\pi : \pi \supseteq S} \left(\prod_{v \in \ve{h(\pi)} \setminus \ve{h'}} \Eb{Y_v^{\tau_{hv}}}\right)
\end{align*}
where $S$ is a set matrix entries.
We can bound the number of such $S$ by $2^{O(q_i)}$ since for each vertex (variable) $v\in h'$ there are only two entries in the matrix $A$ that can be mapped to $v$ by the mapping $h(S)$. Fix an $S$ such that $h(S)=h'$.
Note that whenever $\tau_{hv}=1$ we have $\Eb{Y_v^{\tau_{hv}}}=0$, so we can restrict the sum to be over permutations $\pi$ with $\tau_{hv}=2$ for all $v \in \ve{h(\pi)} \setminus \ve{h'}$. For every fixed $S$, the number of such $\pi$ is at most $n^{(n-q_i)/2}$ since we need to choose the remaining $n-q_i$ entries and each choice fixes two positions in $\pi$. By moment boundedness we have $\left(\prod_{v \in \ve{\pi} \setminus \ve{h'}} \Eb{Y_v^{\tau_{hv}}}\right) \le 2^{(n-q_i)/2}$. We conclude that $w^{(i)}_{h'} \le 2^{O(n)} n^{(n-q_i)/2} $.

Fix some $g^{(i)}$ with total power $q_i\equiv q \le n$. We start from (\ref{inequality3}):
\begin{eqnarray}
\left|\E\left[g^{(i)}(Y)^k\right]\right| & \le & \sum_{\ell=\eta}^{k\eta/2}\frac{1}{\ell!} \sum_{G'\in {\cal S}_2(k,\ell)} \underbrace{\sum_{\pi\in M([\ell])}}_{\le (n^2)^{\ell}}
\underbrace{\left(\prod_{h\in {\cal H}(G')}w_{\pi(h)} \right)}_{\le (2^{O(n)} n^{(n-q)/2})^{k} } \underbrace{\left(\prod_{u\in {\cal V}(G')}\Lambda_{\pi(u)}(G') \right)}_{\le 2^{qk} \prod_{u} D_u!}\nonumber \\
& \le & \sum_{\ell=\eta}^{k\eta/2}\frac{1}{\ell!} \sum_{G'\in {\cal S}_2(k,\ell)} (n^2)^{\ell} (2^{O(n)} n^{(n-q)/2})^{k} \left(2^{qk} \prod_{u\in {\cal V}(G')}  D_u!\right) \nonumber \\
& \le & \sum_{\ell=\eta}^{k\eta/2} 2^{O(n)} R_6^{qk} \frac{1}{\ell!}\left(n^2\right)^\ell k^{qk} n^{k(n-q)/2}  \nonumber \\
& \le &  R_7^{nk} \left(\frac{n^{2\ell}}{\ell^{\ell}}\right) k^{qk} n^{k(n-q)/2} \label{eqn:sunshine}
\end{eqnarray}
where the third inequality bounded $2^{qk}|S_2(\ell)| \prod_{u} D_u!$ by $R_6^{qk}k^{qk}$ using Lemma \ref{MainCount}.   Recall that $\max_{x>0} \left(n k/x\right)^{x} = e^{n k/e}$. We continue using the facts   that we choose $k\le n$ and $\ell \le \eta k/2\le qk/2$,
\begin{eqnarray*}
\left|\E\left[g^{(i)}(Y)^k\right]\right| & \le &  R_7^{nk} \left(\frac{n k}{\ell}\right)^{\ell} \left(\frac{n}{k}\right)^{\ell} k^{qk} n^{k(n-q)/2}\\
& \le &  R_8^{nk}\left(\frac{n}{k}\right)^{\ell} k^{qk} n^{k(n-q)/2}\\
& \le &  R_8^{nk}   n^{qk/2} k^{qk/2} n^{k(n-q)/2}\\
& \le & R_8^{nk} n^{nk/2} k^{nk/2}.
\end{eqnarray*}
Applying Corollary \ref{cor:momentOfSum} yields
\begin{eqnarray*}
\E\left[\left|P(A)\right|^k\right]&\le& \E\left[\left(\sum_{i=1}^{m}\left|g^{(i)}(Y)\right|\right)^k\right] \\
&\le& m^k R_8^{nk} n^{nk/2} k^{nk/2} \\
& \le & R_9^{nk} n^{nk/2} k^{nk/2}
\end{eqnarray*}
since $m \le 2n^2 \le 10^n$ and where $R_9$ is an absolute constant.

This gives us a bound on $\left|\E\left[P(A)^k\right]\right|$ identical to that in the proof of Theorem \ref{perm}, so we finish the proof identically to the proof of Theorem \ref{perm}.
\end{proof}

\section{Examples of Moment Bounded Random Variables}\label{sec:exampleMom}

In this section we show that three classes of random variables are moment bounded and give examples from each class. The classes are bounded random variables,  log-concave continuous random variables, and  log-concave discrete random variables.

\subsection{Bounded random variables}

\begin{lemma}\label{lem:bmb}
Any random variable $Z$ with $|Z| \le L$ is moment bounded with parameter $L$.
\end{lemma}
\begin{proof}
For any $i \ge 1$ we clearly have $|Z|^i \le L |Z|^{i-1}$ hence $\Eb{|Z|^i} \le L \Eb{|Z|^{i-1}} \le i L \Eb{|Z|^{i-1}}$.
\end{proof}

In particular Lemma \ref{lem:bmb} implies that 0/1 and -1/1 random variables are moment bounded with parameter 1.

\subsection{Log-concave continuous random variables}

We say that non-negative function $f$ is \emph{log-concave} if $f(\lambda x + (1-\lambda) y) \ge f(x)^\lambda f(y)^{1-\lambda}$ for any $0 \le \lambda \le 1$ and $x,y \in \R$ (see \cite{BV} Section 3.5). Equivalently $f$ is log concave if $\ln f(x)$ is concave on the set $\{x : f(x)>0\}$ where $\ln f(x)$ is defined and this set is a convex set (i.e.\ an interval).
A continuous random variable (or a continuous distribution) with density $f$ is \emph{log-concave} if $f$ is a log-concave function. See \cite{An,An1,BB, BV} for introductions to log-concavity.

\begin{lemma}\label{lem:lcmb} 
Any non-negative log-concave random variable $X$ with density $f$ is moment bounded with parameter $L=\Eb{X}$.
\end{lemma}
\begin{proof}
Let $\ell = \inf \{x \ge 0 : f(x)>0\}$ and $u=\sup \{x \ge 0 : f(x)>0\}$.  By log-concavity we have that $f(x)>0$ for all $\ell <x<u$.
Let $F(x)=\pr{X \le x}$ and $\bar F(x) = \pr{X \ge x}$. Note that $\bar F(x)=0$ for all $x\ge u$. For any $i \ge 1$ we write
\begin{align}
\Eb{X^i} & = \int_{x=0}^\infty x^i dF(x) \nonumber \\
 & = -\int_{x=0}^\infty x^i d\bar F(x) \nonumber \\
& = -x^i \bar F(x)\Big|_{x=0}^\infty + \int_{x=0}^\infty \bar F(x) d(x^{i}) \nonumber \\
& = 0 + \int_0^\infty \bar F(x) i x^{i-1} dx \nonumber \\
& = \int_0^\ell i x^{i-1} dx + \int_\ell^u \frac{\bar F(x)}{f(x)} i x^{i-1} f(x) dx \label{eqn:momentLC}
\end{align}
where the third equality is integration by parts. It is known (see for example implication B of Proposition 1 in \cite{An} or Theorem 2 in \cite{BB}) that log-concavity of density $f$ implies log-concavity of $\bar F(x)$. It follows that $d (\ln \bar F(x)) / dx = -f(x)/\bar F(x)$ is a non-increasing function of $x$ on $(\ell, u)$ and hence $\bar F(x)/f(x)$ is also non-increasing. It follows that $\frac{\bar F(x)}{f(x)} \cdot i x^{i-1}$ is a product of a non-increasing function and a non-decreasing function. We apply Chebyshev's integral inequality, yielding,
\begin{align*}
\Eb{X^i} & \le \int_0^\ell i x^{i-1} dx + \left[\int_{\ell}^u \frac{\bar F(x)}{f(x)} f(x)dx \right] \left[\int_{\ell}^u i x^{i-1} f(x)dx \right] \\
& = \int_0^\ell i x^{i-1} dx + \left[\int_{\ell}^u {\bar F(x)} dx \right] \cdot i\Eb{X^{i-1}} \\
& = \ell^i + \left[\int_{\ell}^u {\bar F(x)} dx \right] \cdot i\Eb{X^{i-1}} \\
& \le i \ell \Eb{X^{i-1}} + \left[\int_{\ell}^u {\bar F(x)} dx \right]\cdot i\Eb{X^{i-1}} \\
& =\left[\int_{0}^{+\infty} {\bar F(x)} dx \right]\cdot i\Eb{X^{i-1}} = \Eb{X} \cdot i\Eb{X^{i-1}}
\end{align*}
where we used the fact $\Eb{|X|}=\int_{0}^{+\infty} {\bar F(x)} dx$.
\end{proof}

\begin{lemma}\label{lem:lcmb2}
Any log-concave random variable $X$ with density $f$ is moment bounded with parameter $L=\frac{1}{\ln 2}\Eb{|X|} \approx 1.44 \Eb{|X|}$.
\end{lemma}
\begin{proof}
If $X$ is non-negative or non-positive with probability 1 the Lemma follows from Lemma \ref{lem:lcmb}, so suppose not.
Write
\begin{align*}
\Eb{|X|^k} &= \pr{X \ge 0}\Eb{X^k | X \ge 0} + \pr{X < 0}\Eb{(-X)^k | X < 0} \\
&= \pr{X \ge 0}\Eb{X_+^k} + \pr{X < 0}\Eb{X_{-}^k}
\end{align*}
where $X_+$ (resp.\ $X_-$) is a non-negative random variable with density at $x$ proportional to $f(x)$ (resp.\ $f(-x)$) for $x \ge 0$ and zero for $x < 0$. Clearly $X_+$ and $X_-$ are log-concave. Lemma \ref{lem:lcmb} yields
\begin{align*}
\Eb{|X|^k} &\le \pr{X \ge 0}k \Eb{|X_+|} \Eb{|X_+|^{k-1}} + \pr{X < 0}k \Eb{|X_-|} \Eb{|X_-|^{k-1}} \\
&\le k \max\{\Eb{|X_+|}, \Eb{|X_-|}\} \left(\pr{X \ge 0}\Eb{|X|^{k-1} | X \ge 0} + \pr{X < 0}\Eb{|X|^{k-1} | X < 0}\right) \\
& = k \max\{\Eb{|X|~|~X \ge 0}, \Eb{|X|~|~X < 0}\} \Eb{|X|^{k-1}} \\
& \le k \frac{1}{\ln 2}\Eb{|X|} \Eb{|X|^{k-1}}
\end{align*}
where we used Lemma \ref{lem:uncondition} in the last inequality to bound $\max \{\Eb{|X| ~|~ X \le 0}, \Eb{|X| ~|~ X \ge 0}\} \le \frac{1}{\ln 2} \Eb{|X|}$.
\end{proof}

The survey \cite{BB} lists many distributions with log-concave densities: normal, exponential, logistic, extreme value, chi-square, chi, Laplace, Weibull, Gamma, and Beta, where the last three are log-concave only for some parameter values. Lemma \ref{lem:lcmb2} implies that random variables with any of these distributions are moment bounded.

Any random variable trivially satisfies $\Eb{|X|^1} = 1 \Eb{|X|}\Eb{|X|^{1-1}}$ so for every random variable (that is non-zero with positive probability) Lemma \ref{lem:lcmb} gives the smallest possible moment boundedness parameter $L$ and Lemma \ref{lem:lcmb2} gives $L$ that is within a factor of $1/\ln 2$ of the best possible. An exponentially distributed random variable is tight for Lemma \ref{lem:lcmb} in an even stronger sense: $\Eb{|X|^k} = k \Eb{|X|}\Eb{|X|^{k-1}}$ for \emph{all} integers $k \ge 1$.

The following example shows that Lemma \ref{lem:lcmb2} is tight. Let $X$ have density
\[
f(x) = \piecewise{e^{(x-x_0)} & \tif x \le x_0 \\ 0 & \tif x > x_0}
\]
where $x_0 = \ln 2$. This density is clearly log-concave. Using integration by parts we derive
\begin{eqnarray*}
\Eb{|X|} &=& \int_{0}^{+\infty}xe^{-x-x_0}dx +\int_{0}^{\ln 2}xe^{x-x_0} dx \\
&=&-\frac{xe^{-x}}{2}\Big|_0^{+\infty}+\int_{0}^{+\infty}\frac{e^{-x}}{2}dx +\frac{xe^{x}}{2}\Big|_0^{\ln 2}-\int_{0}^{\ln 2}\frac{e^{x}}{2} dx\\
&=&0+1/2+\ln 2-1/2=\ln 2.
\end{eqnarray*}
 The $k$-th moment for large $k$ is dominated by the exponential left tail: $\Eb{|X|^k} = e^{-x_0} k! + O(1)$. Therefore $\lim_{k \to \infty} \frac{\Eb{|X|^k}}{\Eb{|X|^{k-1}}} = k = k \frac{1}{\ln 2} \Eb{|X|}$, hence $X$ is moment-bounded for no $L <\frac{1}{\ln 2} \Eb{|X|}$.

\subsubsection{Technical lemmas}

This section is devoted to proving the following Lemma.

\begin{lemma}\label{lem:uncondition}
For any random variable $X$ with log-concave density, $\pr{X \ge 0}>0$ and $\pr{X \le 0} > 0$ we have
\[
 \max \{\Eb{|X| ~|~ X \le 0}, \Eb{|X| ~|~ X \ge 0}\} \le \frac{1}{\ln 2} \Eb{|X|} \approx 1.44 \Eb{|X|}
.\]
\end{lemma}

The following Lemma about log-concave functions is intuitive but a bit technical to prove.
\begin{lemma}\label{lem:logconcave}
Suppose $f$ is a log-concave function, $x_0 >0$, $h(x)=\piecewise{e^{x-x_0} & \tif x \le x_0 \\ 0 & \tif x > x_0}$, $f(0)=h(0)$, and $\int_{-\infty}^0 f(x) \mathrm{d}x = \int_{-\infty}^0 h(x) \mathrm{d}x$.
 It follows that:
\begin{enumerate}
\item there exists $x_1<0$ such that $(x-x_1)(f(x)-h(x)) \ge 0$ for any $x \le 0$ and
\item $(x-x_0)(f(x)-h(x)) \ge 0$ for any $x \ge 0$.
\end{enumerate}
\end{lemma}
\begin{proof}
Let $S^+ = \{x \le 0 : f(x) > h(x) \}$ and $S^- = \{x \le 0 : f(x) < h(x) \}$. We have
\[
0 = \int_{-\infty}^0f(x)dx - \int_{-\infty}^0h(x)dx = \int_{S^+} (f(x)-h(x))\mathrm{d}x - \int_{S^-} (h(x)-f(x))\mathrm{d}x
\]
and an integral of a positive function is positive iff it is over a set of positive measure, hence either both $S^+$ and $S^-$ have Lebesgue measure zero or neither do.
Below we will use the following simple fact about concave functions. If $g(x)$ is concave, $g(z)=0$, $g(z')>0$ and $z'<z$ then $g(z'')>0$ for all $z''\in (z,z)$.

The log-concavity of $f$ implies that $\ln f(x) - \ln h(x)$ is concave on $(-\infty, x_0]$. This and the fact that $f(0)=h(0)$ imply  the following key properties: if $x \in S^+$ then $S^+ \supseteq [x,0)$ and similarly if $y \in S^-$ then $S^- \supseteq (-\infty, y]$. Among other things these properties imply that if $S^+$ (resp.\ $S^-$) is non-empty it contains an interval and hence has positive measure.
If both $S^+$ and $S^-$ have measure zero then $f(x)=h(x)$ for all $x \le 0$ and any $x_1 < 0$ will satisfy the first part of the lemma.
If both $S^+$ and $S^-$ have positive measure then $x_1=\inf S^+$ will satisfy the first part of the lemma.

The first part of the lemma implies that $x_2=x_1/2<0$ satisfies $f(x_2) \ge h(x_2)$.
For any $0 < x < x_0$ the facts that $\ln f(x) - \ln h(x)$ is concave on $(-\infty, x_0]$, $f(x_2) \ge h(x_2)$ and $f(0)=h(0)$ imply that $f(x) \le h(x)$. Clearly $f(x) \ge h(x)=0$ for $x > x_0$, so the second part of the Lemma follows.
\end{proof}

Now we prove Lemma \ref{lem:uncondition}.
\begin{proof}
Let random variable $X$ be given with density $f$. We prove the upper-bound on $\Eb{|X| ~|~ X \le 0}$ only; the upper-bound on $\Eb{|X| ~|~ X \ge 0}$ follows from this bound because $-X$ is log-concave. The Lemma is invariant with respect to scaling $X$ so we can and do assume without loss of generality that $f(0)=\pr{X \le 0}$.

Let $X'$ be a random variable with density $h(x)=\piecewise{e^{x-x_0} & \tif x \le x_0 \\ 0 & \tif x > x_0}$ where $x_0$ is the solution to $e^{-x_0}=\pr{X \le 0}$. One can readily verify that $\pr{X' \le 0} = e^{-x_0} = \pr{X \le 0}$. Therefore,  $\pr{X' \ge 0} = \pr{X \ge 0}$, and $h(0)=f(0)$.

Using the first part of Lemma \ref{lem:logconcave} we have
\begin{align*}
0 & \le \int_{-\infty}^0  (x-x_1)(f(x)-h(x))dx \\
&= \int_{-\infty}^0 x (f(x)-h(x)) dx - x_1(\pr{X \le 0} - \pr{X' \le 0}) \\
& = \pr{X \le 0}(-\Eb{|X|~|~X \le 0} + \Eb{|X'|~|~X' \le 0}) - 0
\end{align*}
i.e.\ $\Eb{|X'|~|~X' \le 0} \ge \Eb{|X|~|~X \le 0}$.

By the second part of Lemma \ref{lem:logconcave} we have
\begin{align*}
0 &\le \int_{0}^{+\infty}  (x-x_0)(f(x)-h(x))dx \\
&= \int_{0}^{+\infty} x (f(x)-h(x)) dx - x_0(\pr{X \ge 0} - \pr{X' \ge 0}) \\
&= \pr{X \ge 0}(\Eb{|X|~|~X \ge 0}-\Eb{|X'|~|~X' \ge 0} ) - 0
\end{align*}
i.e.\ $\Eb{|X'|~|~X' \ge 0} \le \Eb{|X|~|~X \ge 0}$.

We conclude that
\begin{align}
\frac{\Eb{|X|}}{\Eb{|X|~|~ X \le 0}} &= \frac{\pr{X \le 0}\Eb{|X| ~|~ X \le 0}  + \pr{X \ge 0}\Eb{|X| ~|~ X \ge 0}}{\Eb{|X|~|~ X \le 0}} \nonumber \\
&= \pr{X \le 0} + \pr{X \ge 0}\cdot \frac{\Eb{|X|~|~ X \ge 0}}{\Eb{|X|~|~ X \le 0}} \nonumber \\
&\ge \pr{X' \le 0} + \pr{X' \ge 0}\cdot \frac{\Eb{|X'|~|~ X' \ge 0}}{\Eb{|X'|~|~ X' \le 0}} \nonumber \\
& = \pr{X' \le 0} + \pr{X' \le 0}\frac{\pr{X' \ge 0}\Eb{|X'|~|~ X' \ge 0}}{ \pr{X' \le 0}\Eb{|X'|~|~ X' \le 0}} \nonumber \\ 
& = e^{-x_0} + e^{-x_0}\frac{x_0 + e^{-x_0} - 1}{e^{-x_0}} \nonumber \\
& = x_0 + 2e^{-x_0} - 1 \equiv r(x_0)
\end{align}
where we used integration by parts to compute
\begin{eqnarray*}
\pr{X' \le 0}\Eb{|X'| ~|~ X' \le 0} &=& \int_{-\infty}^{0} (-x) e^{x-x_0} dx \\
&=& (-x)e^{x-x_0}|^0_{-\infty} - \int_{-\infty}^{0} (-1) e^{x-x_0} dx = 0 + e^{- x_0} \\
\pr{X' \ge 0}\Eb{|X'| ~|~ X' \ge 0} &=& \int_{0}^{x_0} x e^{x-x_0} dx \\
&=& x e^{x-x_0}|^{x_0}_{0} - \int_{0}^{x_0} e^{x-x_0} dx= x_0 + e^{-x_0} - 1.
\end{eqnarray*}
Finally we note that $r(x_0)$ is minimized on $0 \le x_0 < \infty$ when $0=\frac{dr}{dx_0}=1 - 2 e^{-x_0}$. We conclude that
\[
\frac{\Eb{|X|}}{\Eb{|X|~|~ X \le 0}} \ge r(x_0) \ge r(\ln 2) = \ln 2
\]
which implies the Lemma.
\end{proof}

\subsection{Log-concave discrete random variables}

A distribution over the integers $\dots, p_{-2},p_{-1},p_0,p_1,p_2,\dots$ is said to be \emph{log-concave} \cite{An,JK} if $p_{i+1}^2 \ge p_i p_{i+2}$ for all $i$.
An integer-valued random variable $X$ is log-concave if its distribution $p_i = \pr{X=i}$ is.

\begin{lemma}\label{lem:dlcmb}
Any non-negative integer-valued log-concave random variable $X$ is moment bounded with parameter $L=1+\Eb{|X|}$.
\end{lemma}
\begin{proof}
The proof parallels the proof of Lemma \ref{lem:lcmb}. Let $r_i=\pr{X\ge i}=\sum_{j=i}^\infty p_j$, $\ell=\min \{i:p_i>0\}$ and $u=\max \{i:p_i > 0\}$.
For any $k \ge 1$ we have
\begin{align}
\Eb{|X|^k} &= \sum_{x=0}^\infty p_x x^k \nonumber \\
&= \sum_{x=1}^\infty (r_x-r_{x+1}) x^k \nonumber \\
&= \sum_{x=1}^\infty r_x (x^k - (x-1)^k) \nonumber \\
& \le \sum_{x=1}^\infty r_x k x^{k-1} = \sum_{x=0}^\infty r_x k x^{k-1} \nonumber\\
& = \sum_{x=0}^{\ell-1} k x^{k-1} + \sum_{x=\ell}^u \frac{r_x}{p_x} k x^{k-1} p_x \nonumber\\
& \le \sum_{x=0}^{\ell-1}k x^{k-1} + \left( \sum_{x=\ell}^u \frac{r_x}{p_x} p_x\right) \left( \sum_{x=\ell}^u k x^{k-1} p_x\right) \nonumber\\
& \le \max\{0,\ell-1\} \Eb{k|X|^{k-1}} + \left(\Eb{|X|}+1-\ell\right) \Eb{k |X|^{k-1}} \nonumber\\ 
& = \left(1+\Eb{|X|}\right) \Eb{k |X|^{k-1}} \nonumber
\end{align}
where the second inequality uses the fact that $\frac{r_x}{p_x}$ is a non-increasing sequence (Proposition 10 in \cite{An}) and Chebyshev's sum inequality.
\end{proof}

\begin{lemma}\label{lem:dlcmb2}
Any log-concave integer-valued random variable $X$ is moment bounded with parameter $L=\max(\Eb{|X| ~|~ X \ge 0}, \Eb{|X| ~|~ X < 0})$.
\end{lemma}

We omit the proof of Lemma \ref{lem:dlcmb2}, which is almost identical to the proof of Lemma \ref{lem:lcmb2}.

Examples of log-concave integer-valued distributions include Poisson, binomial, negative binomial and hypergeometric  \cite{An,JK}. Random variables with these distributions are moment bounded by Lemma \ref{lem:dlcmb}.

The parameter $1+\Eb{|X|}$ in Lemma \ref{lem:dlcmb} cannot be improved to match the $\Eb{|X|}$ in Lemma \ref{lem:lcmb}. Indeed a Poisson distributed random variable with mean $\mu$ has $\Eb{X^2} = \mu^2 + \mu$, which exceeds the desired bound of $2\Eb{X}\Eb{X^{2-1}}=2\mu^2$ when $\mu < 1$.


\section{Examples Showing Tightness of the Bounds}\label{tight}

This section deals exclusively with multilinear polynomials with non-negative coefficients over independent 0/1 random variables. We use notation specialized to this case: for a polynomial $f(x)$ and 0/1 random variables $Y_1,\dots,Y_n$ we have
\[
\mu_r(f,Y) = \max_{A\subseteq {\cal V}: |A|=r}\left\{ \sum_{h\in {\cal  H} | A\subseteq h } w_h\prod_{i\in h\setminus A}\Eb{Y_i}\right\}
.\]
We continue to omit the $Y$ from $\mu_r(f,Y)$ when it is clear from context.

\begin{lemma}\label{lem:muProd}
We are given a power $q_1$ polynomial $f_1(x)$ with corresponding hypergraph $H_1$, weights $w_1$, vertices ${\cal V}(H_1)=[n]$ and a power $q_2$ polynomial $f_2(x)$ with corresponding hypergraph $H_2$, weights $w_2$, vertices  ${\cal V}(H_2)=\{n+1,\dots,m\}$. We are also given $m$ independent 0/1 random variables $X_1,\dots,X_m$.
Then the product polynomial $fg=(H, w)$ defined by $(fg)(x_1,\dots,x_m)=f(x_1,\dots,x_n)g(x_{n+1},\dots,x_m)$ satisfies
\[
\mu_i(fg,X)) = \max_{0 \le i_1 \le q_1 : 0 \le i - i_1 \le q_2} \mu_{i_1}(f,X_1,\dots,X_n) \mu_{i - i_1}(g,X_{n+1},\dots,X_m)
.\]
\end{lemma}

\begin{proof}
The Lemma follows easily from the definition of $\mu_i$ and the fact that restriction to hyperedges containing a fixed set of vertices preserves the product structure. Indeed let ${\cal H}_1={\cal H}(H_1)$, ${\cal H}_2={\cal H}(H_2)$ and ${\cal H}={\cal H}(H)$. Then
\begin{align*}
\mu_i(fg)
& = \max_{A\subseteq {\cal V}: |A|=r}\left\{ \sum_{h\in {\cal  H} | A\subseteq h } w_h\prod_{v\in h\setminus A}\Eb{X_v}\right\} \\
& = \max_{\substack{0 \le i_1 \le q_1, 0 \le i_2 \le q_2 :\\ i_1 + i_2 = i}} ~
		\max_{\substack{A_1\subseteq {\cal V}_1: \\ |A_1|=i_1}} ~ \max_{\substack{A_2\subseteq {\cal V}_2: \\ |A_2|=i_2}}
		\left\{ \sum_{\substack{h_1\in {\cal  H}_1 : \\ A_1\subseteq h_1 }} ~ \sum_{\substack{h_2\in {\cal  H}_2 : \\ A_2\subseteq h_2 }} w_{h_1} w_{h_2}
		\prod_{v\in h_1\setminus A_1} \Eb{X_v} \prod_{v\in h_2\setminus A_2} \Eb{X_v}
		\right\} \\
& = \max_{\substack{0 \le i_1 \le q_1, 0 \le i_2 \le q_2 :\\ i_1 + i_2 = i}}
		\left(\max_{\substack{A_1\subseteq {\cal V}_1: \\ |A_1|=i_1}} \sum_{\substack{h_1\in {\cal  H}_1 : \\ A_1\subseteq h_1 }} \prod_{v\in h_1\setminus A_1} \Eb{X_v}\right)
		\left(\max_{\substack{A_2\subseteq {\cal V}_2: \\ |A_2|=i_2}} \sum_{\substack{h_2\in {\cal  H}_2 : \\ A_2\subseteq h_2 }} \prod_{v\in h_2\setminus A_2} \Eb{X_v}\right) \\
& = \max_{\substack{0 \le i_1 \le q_1, 0 \le i_2 \le q_2 :\\ i_1 + i_2 = i}}
		\mu_{i_1}(f) \mu_{i_2}(g)
.\end{align*}
\end{proof}

Our next lemma studies a particular sort of complete multilinear $q$-uniform polynomials that we will use frequently.
\begin{lemma}\label{lem:completeMultilinear}
Given $Z = \binom{\sum_{i=1}^n X_i}{q} = \sum_{h \subseteq [n] : |h| = q} \prod_{v \in h} X_v$ where the $X_v$ are independent 0/1 random variables with $\Eb{X_i} = p \le 0.5$ we have
\begin{itemize}
\item $\mu_i(Z) = \binom{n-i}{q-i} p^{q-i} \le (np)^{q-i}$ and
\item  $\pr{Z = \binom{c}{q}} = \binom{n}{c} p^{c} (1-p)^{n-c} \ge e^{-2np} (\frac{np}{c})^c$ for any integer $0 \le c \le n$.
\end{itemize}
\end{lemma}

\begin{proof}
The first is immediate from definitions. The second follows because
\[
\pr{Z = \binom{c}{q}} = \pr{\sum_i X_i = c} = \binom{n}{c} p^{c} (1-p)^{n-c} \ge (n/c)^c p^c e^{(n-c) \ln (1-p)} \ge \left(np/c\right)^c e^{-2np}
\]
where we used $\ln(1-p) \ge -2p$ for $0 \le p \le 1/2$ in the last inequality.
\end{proof}

\begin{lemma}\label{lem:LBOne}
For any $q \in \N$, $0<\epsilon \le 1$, $\lambda > \mu_q^* > 0$, there is a non-negative power $q$ polynomial $f(x)$ and independent 0/1 random variables $X_1,\dots,X_m$ such that
\begin{itemize}
\item $\mu_j(f,X) \le \epsilon^{q-j} \mu_q^*$ for all $0 \le j \le q$.
\item $\pr{f(X) - \Eb{f(X)} \ge \lambda} \ge \exp\left\{-2 \epsilon\right\} \left( \frac{\epsilon}{4 q (\lambda/\mu_q^*)^{1/q}} \right)^{4 q (\lambda/\mu_q^*)^{1/q}}$.
\item $\pr{f(X) - \Eb{f(X)} \ge \mu_q^*} \ge \exp\left\{-2 \epsilon\right\}  \left(\frac{\epsilon}{q+1}\right)^{q+1}$.
\end{itemize}
\end{lemma}

\begin{proof}
We pick $f(x) = \mu_q^*\cdot \sum_{I \subseteq M, |I| = q} \prod_{i \in I} x_i$ where $|M|=m=\ceil{4 q (\lambda/\mu_q^*)^{1/q}}$ and each $X_i$ is 1 with probability $\epsilon/m \le 1/2$. By Lemma \ref{lem:completeMultilinear} we have
\[
\mu_j(f) \le \left(m \cdot \frac{\epsilon}{m}\right)^{q-j}\mu_q^* = \epsilon^{q-j}\mu_q^*.
\]

The third part of the lemma follows from Lemma \ref{lem:completeMultilinear} with $c=q+1$. Indeed $\binom{q+1}{q} =q+1\ge 2 \ge \epsilon^q+1 \ge (\Eb{f} + \mu_q^*)/\mu_q^*$. Therefore,
$$\pr{f(X) - \Eb{f(X)} \ge \mu_q^*} \ge \pr{f(X) =\mu_q^*\binom{c}{q}}\ge  e^{-2\epsilon} \left(\frac{\epsilon}{c}\right)^c  =
\exp\left\{-2 \epsilon\right\}  \left(\frac{\epsilon}{q+1}\right)^{q+1}.$$

Towards proving the second part of the lemma choose $c$ such that
\[
\binom{c-1}{q} \le \frac{\lambda + \Eb{f}}{\mu_q^*} < \binom{c}{q}
.\]
We have $\lambda + \Eb{f} \ge \mu_q^*$ so such a $c \ge q+1$ exists.
Note that
\[
\frac{(c-1)^q}{q^q} \le \binom{c-1}{q} \le \frac{\lambda + \Eb{f}}{\mu_q^*} \le \frac{\lambda+\mu_q^*}{\mu_q^*} \le \frac{2\lambda}{\mu_q^*}
\]
hence
\[
c \le 1 + q (2 \lambda/\mu_q^*)^{1/q} \le 4 q (\lambda/\mu_q^*)^{1/q} \le m
.\]

By the second part of Lemma \ref{lem:completeMultilinear} we have
\begin{align*}
\pr{f(X)- \Eb{f(X)} \ge \lambda} & \ge
\pr{f(X) = \mu_q^*\binom{c}{q}} \ge e^{-2\epsilon} \left(\frac{\epsilon}{c}\right)^c
\ge e^{-2\epsilon} \left( \frac{\epsilon}{4 q (\lambda/\mu_q^*)^{1/q}} \right)^{4 q (\lambda/\mu_q^*)^{1/q}}
.\end{align*}
\end{proof}

The \emph{binomial distribution} $B(n,p)$ is the distribution of the sum of $n$ independent 0/1 random variables each with mean $p$.
The following lower bound on concentration of binomially distributed random variables is well known (e.g. \cite{F} has more general and precise bounds) but we include a proof for completeness in the Appendix.
\begin{lemma}\label{claim:chernoffTight}
For any $\mu  \ge 27$ and $0 < \lambda \le \mu$ there exists a binomially distributed random variable $Z$ with $\Eb{Z} =\mu$ and
\begin{align}
\pr{Z \ge \Eb{Z} + \lambda} &\ge e^{-100-\frac{\lambda^2}{\Eb{Z}}} \label{eqn:binomLB}
.\end{align}
\end{lemma}

{\bf Remark:} The restriction that $\mu$ is bounded away from zero is needed since when $\mu = \lambda \ll 1$ the right hand side of (\ref{eqn:binomLB}) is constant and the left hand side is necessarily small because $\pr{Z \ge \Eb{Z} + \lambda} = \pr{Z \ge 1} \le \Eb{Z} = \mu$.

\begin{lemma}\label{lem:LBTwo}
For any $q \in \N$, $\mu_q^* > 0$, $\mu_0^* \ge 27\mu_q^*$, $0 < \lambda \le \mu_0^*$ and $0<\epsilon \le 1$ there is a polynomial $f$ of power $q$ and independent 0/1 random variables $X_1,\dots,X_m$ such that
\begin{itemize}
\item $\mu_0(f) = \mu_0^*$,
\item $\mu_q(f) = \mu_q^*$,
\item $\mu_j(f) \le \epsilon \mu_q^*$ for all $1 \le j \le q-1$,
\item $\pr{f(X) - \Eb{f(X)} \ge \lambda} \ge e^{-100} e^{-\frac{\lambda^2}{\mu_0^* \mu_q^*}}$.
\end{itemize}
\end{lemma}

\begin{proof}
We fix a sufficiently large integer $n$ such that $\left(\frac{\mu_0^*}{n \mu_q^*}\right)^{1/q} \le  \epsilon$ and pick our polynomial $f(X)$ to be essentially a linear function in disguise:
\[
f(X) = \mu_q^* \cdot \sum_{0 \le i \le n-1} X_{qi+1}\cdot X_{qi+2}\cdot\ldots\cdot X_{qi+q}
\]
where  $X_i$ are boolean random variables with $Pr[X_i=1]= \left(\frac{\mu_0^*}{n \mu_q^*}\right)^{1/q} \le  \epsilon$. Observe that $\mu_q(f)=\mu_q^*$, $\mu_0(f)=\mu_0^*$, and for $1 \le i \le q-1$ we have $\mu_i(f) = (\frac{\mu_0}{n\mu_q^*})^{(q-i)/q} \mu_q^*\le \epsilon^{q-i} \mu_q^*\le \epsilon \mu_q^*$.

Observe that $f(X)/\mu_q^*$ has the same distribution as a binomially distributed random variable with mean $\mu_0^*/\mu_q^* \ge 27$. The lower-bound on $\pr{f(X) - \Eb{f(X)} \ge \lambda}$ therefore follows from Lemma \ref{claim:chernoffTight} for sufficiently large $n$.
\end{proof}

The following lemma shows how to use a counterexample polynomial of power less than $q$ in place of a counterexample of power $q$.

\begin{lemma}\label{lem:increaseDegree}
For any $\epsilon>0$, $q$-uniform hypergraph $H=({\cal V},{\cal  H})$, non-negative weights $w$, polynomial $f(x)=\sum_{h\in {\cal H}} w_h\prod_{v\in h}x_v$, independent 0/1 random variables $X_1,\dots,X_n$
 and any $q'>q$ there exists a $q'$-uniform hypergraph $H=({\cal V'},{\cal  H'})$, non-negative weights $w'$, independent 0/1 random variables $X_1',\dots,X_{n'}'$ and polynomial $f'(x')=\sum_{h'\in {\cal H}'} w'_{h'}\prod_{i\in h'}x'_i$ such that
\begin{itemize}
\item $\mu_i(f') \le \mu_i(f)$ for all $i \le q$,
\item $\mu_i(f') \le \epsilon \mu_q(f)$ for all $q < i \le q'$,
\item $\pr{f'(X') - \Eb{f'(X')} \ge \lambda} \ge 2^{-(q'-q)}\pr{f(X) - \Eb{f(X)} \ge \lambda}$.
\end{itemize}
\end{lemma}

\begin{proof}
We let $f'(X,Y) = f(X)g(Y)$ where $g(Y)$ is a power-$(q'-q)$ polynomial that is well concentrated around 1. In particular we use
\[
g(Y) = \left(\frac{2}{m} \sum_{i=1}^m Y_{1,i} \right)\cdot \left(\frac{2}{m} \sum_{i=1}^m Y_{2,i} \right) \cdot\ldots\cdot \left(\frac{2}{m} \sum_{i=1}^m Y_{q'-q,i} \right)
\]
where the $n'=(q'-q)m$ random variables $Y_{ij}$ are independent with mean 1/2 and
\[
m=\ceil{\max \left(2/\epsilon, 2\max_{i,j} \frac{\mu_j(f, X)}{\mu_i(f, X)}\right)}
.\]
Note that $2/m \le \epsilon$ and $2/m \le \frac{\mu_j(f, X)}{\mu_i(f, X)}$ for any $0 \le i,j \le q$.  It is easy to see that $\mu_i(g, Y) = (2/m)^{i}$. Let $X'=X_1,\dots,X_n,Y_{1,1},\dots,Y_{q'-q,m}$ denote the random variables that $f'$ is a function of.
By Lemma \ref{lem:muProd} we get that
\begin{eqnarray}
 \mu_i(f', X') &=& \max_{0 \le j \le q : 0 \le i - j \le q'} \mu_j(f) \mu_{i-j}(g) \nonumber\\
 &=& \max_{0 \le j \le q : 0 \le i - j \le q'} \mu_j(f,X) (m/2)^{-(i-j)} \nonumber\\
 &=& \max_{0 \le j \le \min(q, i)} B_{ij} \label{eqn:mufg}
\end{eqnarray}
where $B_{ij}=\mu_j(f,X) (m/2)^{-(i-j)}$. We bound (\ref{eqn:mufg}) in two cases. The first case is $i \le q$. For any $0 \le j < i \le q$ we have
\begin{align}
B_{ij} &= \mu_j(f,X) (m/2)^{-(i-j)} \le \mu_j(f,X) (2/m) \le \mu_i(f,X) \label{eqn:ilej}
.\end{align}
Clearly $B_{ij} \le \mu_i(f,X)$ holds for $i=j$ as well, so we conclude that $\max_{0 \le j \le \min(q, i)} B_{ij} \le \mu_i(f,X)$ when $i \le q$.

The other case is when $i>q$. For any $0 \le j < q <i$ we have
\begin{align}
\mu_j(f,X) (m/2)^{-(i-j)} &\le \mu_j(f,X) (2/m)^2 \le \mu_j(f,X) \cdot  \frac{\mu_q(f,X)}{\mu_j(f,X)} \cdot \epsilon = \epsilon \mu_q(f,X) \label{eqn:igj}
.\end{align}
Similarly for $0 \le j = q <i$ we have
\begin{align}
\mu_j(f,X) (m/2)^{-(i-j)}&=\mu_q(f,X) (m/2)^{-(i-q)} \le \mu_q(f) (2/m) \le \epsilon \mu_q(f,X) \label{eqn:igj2}
.\end{align}
Combining (\ref{eqn:mufg}), (\ref{eqn:ilej}), (\ref{eqn:igj}) and (\ref{eqn:igj2}) we conclude that
\[
 \mu_i(f', X') \le \left\{\begin{array}{cc} \mu_{i}(f,X), & \text{if } i \le q, \\
 								\epsilon \mu_q(f,X), & \text{otherwise.} \end{array} \right.
\]

To show the last part of the lemma we bound
\begin{align*}
\pr{f(X)g(Y) - \Eb{f(X)g(Y)} \ge \lambda}
&\ge
\pr{f(X) - \Eb{f(X)} \ge \lambda \text{ and } g(X) \ge 1} \\
&=
\pr{f(X) - \Eb{f(X)} \ge \lambda} \pr{g(X) \ge 1} \\
& \ge \pr{f(X) - \Eb{f(X)} \ge \lambda} 2^{-(q'-q)}
\end{align*}
where the last inequality follows because each of the linear terms $\left(\frac{2}{m} \sum_{i=1}^m Y_{j,i} \right)$ in the definition of $g(Y)$ is distributed symmetrically about its mean of 1 and hence is at least one with probability at least $1/2$.
\end{proof}

\begin{proof}{\bf of Theorem \ref{thm:LB}.}
Fix $q$, $\lambda$ and $\set{\mu_i^*}_{0 \le i \le q}$. Let $i$ be the dominant term in (\ref{eq:LB}), i.e.\ $i$ minimizes $\min_i(\lambda^2 / (\mu_0^* \mu_i^*), (\lambda/\mu_i^*)^{1/i})$. We consider three cases.

The first case is when $\lambda \le \mu_i^*$. In this case we apply Lemma \ref{lem:LBOne} (third part) to get a power $i$ polynomial and then Lemma \ref{lem:increaseDegree} to convert it into a power $q$ polynomial, using $\epsilon = \min_{0 \le j,j' \le q} \mu_j^* / \mu_{j'}^*$ for both Lemmas. This yields 0/1 random variables $X_1,\dots,X_n$ and the desired power $q$ polynomial $f(X)$ with $\mu_j(f,X) \le \mu_j^*$ for $0 \le j \le q$ and
\begin{align}
\pr{f(X) - \Eb{f(X)} \ge \lambda} &\ge \pr{f(X) - \Eb{f(X)} \ge \mu_i} \nonumber \\
&\ge 2^{-(q-i)} \exp \left\{-2\epsilon + (i+1) \ln \left(\frac{\epsilon}{i+1}\right)\right\} \nonumber\\
&\ge 2^{-q} \exp \left\{-2 + (q+1) \ln \left(\frac{\epsilon}{q+1}\right)\right\}= \frac{1}{C_1} \nonumber \\
&\ge \max\left\{e^{-\left(\frac{\lambda^2}{\mu^*_0\mu^*_i} + 1\right)\log C_1},e^{-\left(\left(\frac{\lambda}{\mu^*_i }\right)^{1/i} +1\right)\log C_1}\right\} \label{eq:caseOne}
\end{align}
where $C_1=2^{q} e^{2} ((q+1)/\epsilon)^{q+1}$.

The second case is when $\lambda > \mu_i^*$ and $27 \lambda^2 / (\mu_0^* \mu_i^*) \ge (\lambda/\mu_i^*)^{1/i}$.
We apply Lemma \ref{lem:LBOne} (second part) and Lemma \ref{lem:increaseDegree} using $\epsilon = \min_{0 \le j,j' \le q} \mu_j^* / \mu_{j'}^*$ for both, yielding independent 0/1 random variables $X_1,\dots,X_n$ and a degree $q$ polynomial $f(X)$ with $\mu_j(f,X) \le \mu_j^*$ for $0 \le j \le q$ and
\begin{align}
\pr{f(X) - \Eb{f(X)} \ge \lambda} &\ge 2^{-(q-i)} \exp \left\{-2\epsilon + 4 i (\lambda/\mu_i^*)^{1/i} \ln \left(\frac{\epsilon}{4 i (\lambda/\mu_i^*)^{1/i}}\right)\right\} \nonumber\\
& \ge \frac{1}{C_2}(C_3)^{-(\lambda/\mu_i^*)^{1/i}} \nonumber\\
& \ge \frac{1}{C_2}(C_4)^{-\min((\lambda/\mu_i^*)^{1/i}, \lambda^2 / (\mu_0^* \mu_i^*))} \nonumber\\
& \ge \max\left\{e^{-\left(\frac{\lambda^2}{\mu^*_0\mu^*_i} + 1\right)\log (\max \{C_2, C_4 \})},e^{-\left(\left(\frac{\lambda}{\mu^*_i }\right)^{1/i} +1\right)\log (\max \{C_2, C_4 \})}\right\} \label{eq:caseTwo}
\end{align}
for $C_2 = e^2 2^q$, $C_3 = (\lambda/\mu_i^*)^{4} (\frac{4 i}{\epsilon})^{4i}$, and $C_4 = C_3^{27}$.

The final case is when $\lambda > \mu_i^*$ and $27 \lambda^2 / (\mu_0^* \mu_i^*) < (\lambda/\mu_i^*)^{1/i}$.
These constraints imply that $\mu_0^* > 27 \lambda^{2-1/i} (\mu_i^*)^{1/i - 1} \ge 27 \lambda$, hence $\lambda < \mu_0^*$. We also have $\mu_0^* >27 \lambda > 27 \mu_i^*$.
We apply Lemmas \ref{lem:LBTwo} and Lemma \ref{lem:increaseDegree} with $\epsilon = \min_{0 \le j,j' \le q} \mu_j^* / \mu_{j'}^*$ for both, yielding a polynomial $f$ and independent 0/1 random variables $X_1,\dots,X_n$ with $\mu_j(f,X) \le \mu_j^*$ for $0 \le j \le q$ and
\begin{align}
Pr\left[f(X_1,\dots,X_n)\ge \Eb{f} + \lambda\right] &\ge e^{-100} 2^{-(q-i)} e^{-\frac{\lambda^2}{\mu_0\mu_i}} \nonumber \\
&\ge \frac{1}{C_5} e^{-\frac{\lambda^2}{\mu_0\mu_i}} \nonumber \\
& \ge \max\left\{e^{-\left(\frac{\lambda^2}{\mu^*_0\mu^*_i} + 1\right)\log C_5},e^{-\left(\left(\frac{\lambda}{\mu^*_i }\right)^{1/i} +1\right)\log C_5}\right\} \label{eq:caseThree}
\end{align}
where $C_5 = e^{100} \cdot 2^q$. This completes the case analysis.

Let $C = \max \{C_1, C_2, C_4, C_5\} \le c_0\Lambda_1^{c_1}\Lambda_2^{c_2}\Lambda_3^{c_3}$ for appropriate absolute constants $c_0$, $c_1$, $c_2$ and $c_3$ where $\Lambda_1 =\max_{0 \le i,j \le q} (\mu_i^*/\mu_j^*)^q=\epsilon^{-q}$, $\Lambda_2=\max_{1 \le i \le q} \lambda/\mu_i^*$  and $\Lambda_3=q^q$. The Theorem follows from (\ref{eq:caseOne}), (\ref{eq:caseTwo}) and (\ref{eq:caseThree}).
\end{proof}

\section*{Acknowledgments}
We would like to thank an anonymous referee for many insightful and helpful comments.


\appendix

\section{Linear special case}\label{sec:linear}

In this section we give a short proof of the linear case ($q=1$) of Theorem \ref{main1special}. Concentration in this case was already known, but this special case nicely illustrates many of our techniques with minimal technical complications. In this case hyperedges are just single vertices, so to simplify notation we make no reference to hyperedges. The rest of the proofs appear in the full version of the paper.

\medskip

We have $n$ vertices $1,2,\ldots,n$, independent random variables $Y_1,\dots,Y_n$ that are moment bounded with parameter $L$, and weights $w_1,\dots,w_n$. We assume that $\Eb{Y_v}=0$ and $w_v \ge 0$ for all $v \in [n]$. We are looking for concentration of $f(Y) = \sum_{v \in [n]} w_v Y_v$. Our bounds are based on the parameters $\mu_0 = \sum_{v \in [n]} w_v \Eb{|Y_v|}$ and $\mu_1 = \max_{v \in [n]} w_v$.

Fix even integer $k \ge 2$. By linearity of expectation and independence we have
\begin{eqnarray}
\Eb{f(Y)^k} &=& \sum_{v_1,\dots,v_k \in [n]} w_{v_1}\cdot\dots\cdot w_{v_k} \Eb{Y_{v_1}\cdot\dots\cdot Y_{v_k}} \nonumber \\
&=& \sum_{v_1,\dots,v_k \in [n]} w_{v_1}\cdot\dots\cdot w_{v_k} \prod_{v \in \{v_1,\dots,v_k\}} \Eb{Y_{v}^{|\{i \in [k]: v_i=v\}|}} .
   \label{eq:1}
\end{eqnarray}
For conciseness we write the sum over $v_1,\dots,v_k \in [n]$ in (\ref{eq:1}) as a sum over vectors $\bar v \in [n]^k$ (with components $v_1,\dots,v_k$). The sum over $\bar v$ in (\ref{eq:1}) is awkward to bound because it is very inhomogeneous, including e.g.\ both the case when $v_1=\dots=v_k$ and the case that all the $v_i$ are distinct. We deal with this issue as follows.
Intuitively we generate $\bar v$ by first picking the number of distinct vertices $\ell =|\{v_1,\dots,v_k\}| = |\{\bar v\}|$, secondly picking a vector $\bar u \in [\ell]^k$ (with components $u_1,\dots,u_k \in [\ell]$) of artificial vertices, and finally choosing an injective mapping $\pi$ from the artificial vertices $[\ell]$ into the real vertices $[n]$ and letting $v_i = \pi(u_i)$. This process generates each vector $\bar v$ a total of $\ell!$ times since the names of the artificial vertices are arbitrary.
Combining the above with (\ref{eq:1}) we have
\begin{equation}
\Eb{f(Y)^k} = \sum_{\ell=1}^{k} \frac{1}{\ell!} \sum_{\bar u \in [\ell]^k : |\{\bar u\}|=\ell} ~ \sum_{\pi \in M(\ell)}
w_{\pi(u_1)}\cdots w_{\pi(u_k)} \prod_{u \in [\ell]}\Eb{Y_{\pi(u)}^{|\{i \in [k]: u_i=u\}|}}
\label{eq:1p9}
\end{equation}
where $M(\ell)$ denotes the set of all injective functions from $[\ell]$ to $[n]$. We introduce the notation $d_u=d_u(u_1,\dots,u_k)=|\{i \in [k]: u_i=u\}|$ for the power of $Y_{\pi(u)}$ in (\ref{eq:1p9}).  If any $d_u = 1$ we have $\Eb{Y_{\pi(u)}^{d_u}}=0$, so we can limit the sum in (\ref{eq:1p9}) to the set ${\cal S}_2(\ell)$ of vectors $\bar u \in [\ell]^k$ with $d_1(\bar u),\dots,d_\ell(\bar u) \ge 2$. The constraint that $|\{\bar u\}|=\ell$ is clearly satisfied for all $\bar u \in {\cal S}_2(\ell)$ so we can safely drop it. Note that $\sum_{u \in [\ell]} d_u = k$, so we therefore can reduce the range of $\ell$ to $1 \le \ell \le k/2$. Consequently we have
\begin{equation}
\Eb{f(Y)^k} = \sum_{\ell=1}^{k/2} \frac{1}{\ell!} \sum_{\bar u \in {\cal S}_2(\ell)}
\underbrace{ \sum_{\pi \in M(\ell)} w_{\pi(u_1)}\cdots w_{\pi(u_k)} \prod_{u \in [\ell]} \Eb{Y_{\pi(u)}^{d_u(\bar u)}} }_{(*)}
\label{eq:2}
\end{equation}
We bound $(*)$ as
\begin{align}
(*) &= \sum_{\pi \in M(\ell)} w_{\pi(u_1)}\cdot\dots\cdot w_{\pi(u_k)} \prod_{u \in [\ell]} \Eb{Y_{\pi(u)}^{d_{u}(\bar u)}} \nonumber\\
&\le \sum_{\pi \in M(\ell)} w_{\pi(u_1)}\cdot\dots\cdot w_{\pi(u_k)} \prod_{u \in [\ell]} \Eb{|Y_{\pi(u)}|^{d_{u}(\bar u)}} \nonumber\\
&\le \sum_{\pi \in M(\ell)} w_{\pi(u_1)}\cdot\dots\cdot w_{\pi(u_k)} \prod_{u \in [\ell]} L^{d_u-1} \cdot d_u! \cdot \Eb{|Y_{\pi(u)}|} \nonumber\\
&\le \mu_1^{k-\ell} \sum_{\pi \in M(\ell)}\prod_{u \in [\ell]} w_{\pi(u)} L^{d_u-1} \cdot d_u! \cdot \Eb{|Y_{\pi(u)}|} \nonumber\\
&= \mu_1^{k-\ell} L^{k-\ell} \left(\prod_{u \in [\ell]} d_u!\right) \sum_{\pi \in M(\ell)} \prod_{u \in [\ell]} w_{\pi(u)} \Eb{|Y_{\pi(u)}|} \label{eq:2.5}
\end{align}
where the second inequality uses moment boundedness $d_u-1$ times (per $u$) and the third inequality follows because $\mu_1=\max_v w_v$. We now extend the sum over $\pi \in M(\ell)$ in (\ref{eq:2.5}) (adding additional non-negative terms) to include all mappings from $[\ell]$ into $[n]$ injective or not, which enables us to move the sum over $\pi$ inside the product over $u$ as follows:
\begin{align}
(*)
& \le (\mu_1 L)^{k-\ell} \left(\prod_{u \in [\ell]} d_u!\right) \sum_{\pi(1),\dots,\pi(\ell) \in [n]} \prod_{u \in [\ell]} w_{\pi(u)} \Eb{|Y_{\pi(u)}|} \nonumber \\
& =(\mu_1 L)^{k-\ell} \left(\prod_{u \in [\ell]} d_u!\right)  \prod_{u \in [\ell]} \sum_{v \in [n]} w_{v} \Eb{|Y_{v}|} \nonumber \\
    & = (\mu_1 L)^{k-\ell} \left(\prod_{u \in [\ell]} d_u!\right) \left( \sum_{v \in [n]} w_{v} \Eb{|Y_{v}|}\right)^\ell \nonumber\\
    & = (\mu_1 L)^{k-\ell} \left(\prod_{u \in [\ell]} d_u!\right) \mu_0^\ell \label{eq:3}
.\end{align}
Combining (\ref{eq:2}) with (\ref{eq:3}) we get
\begin{align}
\Eb{f(Y)^k} &\le \sum_{\ell=1}^{k/2} \frac{1}{\ell!} \sum_{\bar u \in {\cal S}_2(\ell)}  (\mu_1 L)^{k-\ell} \left(\prod_{u \in [\ell]} d_u!\right) \mu_0^\ell \nonumber\\
& = \sum_{\ell=1}^{k/2} \frac{1}{\ell!}(\mu_1 L)^{k-\ell} \mu_0^\ell \sum_{d_1,\dots,d_\ell \ge 2 : d_1+\dots+d_\ell =k} \underbrace{\sum_{\bar u \in [\ell]^k : d_1(\bar u) = d_1,\dots,d_\ell(\bar u) = d_\ell}   \left(\prod_{u \in [\ell]} d_u!\right)}_{(\dagger)}  \label{eq:3p5}
\end{align}
where the equality groups the sum over $\bar u$ by the value of the $d_u$. Now we claim that the equality $(\dagger)=k!$ follows easily from either an appeal to multinomial coefficients or the following direct argument. Indeed consider $k$ balls of which $d_u$ are labeled $u$ for all $u \in [\ell]$. Each of the $k!$ permutation of the balls induces a vector $\bar u$ of the labels. Every $\bar u \in [\ell]^k : d_1(\bar u) = d_1,\dots,d_\ell(\bar u) = d_\ell$ is produced by exactly $\prod_{u \in [\ell]} d_u!$ permutations of the balls, proving the claim.

Substituting $(\dagger)=k!$ into (\ref{eq:3p5}), bounding the number of different $d_1,\ldots,d_\ell\ge 1$ with $d_1+\dots+d_\ell=k$ by $2^{k}$, and bounding
$k!/\ell! \le R_0^{k} k^{k}/\ell^{\ell}\le R_0^k\left(\max_{l\ge 1}\frac{k^{\ell}}{\ell^{\ell}}\right)k^{k-\ell}\le R_1^k k^{k-\ell}$ for some constants $1<R_0<R_1$ we get
\begin{align}
\Eb{f(Y)^k}& \le \sum_{\ell=1}^{k/2} \frac{1}{\ell!}(\mu_1 L)^{k-\ell} \mu_0^\ell k! 2^k \nonumber\\
& \le (k/2) (2R_1)^k \max_{\ell\in [k/2]} k^{k-\ell} (\mu_1 L)^{k-\ell} \mu_0^\ell \nonumber\\
& = (k/2) (2R_1)^k \max_{\ell\in [k/2]} (k\mu_1 L)^{k-2\ell} (\sqrt{\mu_0 k \mu_1 L})^{2\ell} \nonumber\\
& \le \left(\max\{4R_1k\mu_1 L, 4R_1\sqrt{\mu_0 k \mu_1 L}\}\right)^k  \label{eq:4}
.\end{align}
For any $\lambda>0$ we choose $k$ so that $B \approx \lambda/e$ and apply Markov's inequality, yielding
\begin{align*}
\pr{|f(Y)| \ge \lambda} \le e^{-k} \le e^2 \max\{e^{-\lambda^2/(R L \mu_0 \mu_1)}, e^{-\lambda/(R L \mu_1)}\}
\end{align*}
after some straightforward calculations (see proof of the main Theorem in the full version of the paper) for some absolute constant $R$.

\bigskip

We now sketch the differences between the above linear case and the general case that is proven in the main body of this paper. In the general case the sequences of vertices $v_1,\dots,v_k$ and $u_1,\dots,u_k$ become sequences of hyperedges. The sums over $u \in [\ell]$ remain sums over vertices.

The biggest conceptual difference in the $q>1$ case is that we consider the number of connected components in the sequence of hyperedges that replaces $u_1,\dots,u_k$. Counting the number of sequences of hyperedges with $c$ connected components is substantially trickier than the above bound on $(\dagger)$.

Bounding the equivalent of $(*)$ by a product of various $\mu_i$ is also substantially more involved.

\section{Proof of the Lemma \ref{claim:chernoffTight}}

Let $Y$ be Poisson distributed with $\Eb{Y}=\mu$, i.e.\ $\pr{Y=i} = e^{-\mu} \mu^i / i!$ for non-negative integers $i$.
We will first show that $Y$ satisfies (\ref{eqn:binomLB}) but with a better constant (96 instead of 100). We will then use a limiting argument to prove the lemma.

Let $\delta = \lambda / \mu \le 1$ and $\delta' = \delta + \frac{3\sqrt{3}}{\sqrt{\mu}} \le \delta+1 \le 2$. We will frequently use the facts that $0 \le \delta \le 1$ and $\frac{3\sqrt{3}}{\sqrt{\mu}} \le \delta' \le 2$ without explicit mention. Let $f(t)=\Eb{e^{tY}}$ and $g_a(t)=f(t)e^{-at}$.
We will use Theorem A.2.1 in \cite{alonspencer08} which states that
\begin{align}
\pr{Y > a - u} & \ge e^{-tu} \left[g_a(t) - e^{-\epsilon u} \left(g_a(t+\epsilon) + g_a(t-\epsilon)\right) \right] \label{eqn:A21}
\end{align}
for any $a,u,t,\epsilon \in \R$ with $u,t,\epsilon,t-\epsilon$ all positive.  We choose these parameters as follows: let $a=(1+\delta')\mu$, $u=3\sqrt{3\mu}$, $t=\ln(1+\delta')$ and $\epsilon = \frac{1}{\sqrt{\mu(1+\delta')}} \le \frac{1}{3\sqrt{3} \cdot 1}$. Note that by concavity $t=\ln(1+\delta')\ge \frac{\delta' \ln (1+2)}{2} \ge  \frac{3\sqrt{3}}{\sqrt{\mu}}\cdot \frac{\ln 3}{2} > \frac{1}{\sqrt{\mu}} \ge \epsilon$ hence $t-\epsilon$ is positive as required.

A standard calculation (e.g.\ Lemma 5.3 in \cite{mitzenmacher}) shows that $f(t')=e^{\mu(e^{t'} - 1)}$ and $g_{a}(t')=e^{\mu(e^{t'} - 1) - at'}$. Therefore
\begin{align}
\ln [g_{a}(t)] & = \mu(e^{\ln(1+\delta')} - 1) - \mu(1+\delta')\ln(1+\delta') \nonumber\\
& = \mu(\delta' - (1+\delta')\ln(1+\delta')) \nonumber \\
& \ge -\mu \delta'^2 / 2 \label{eqn:g}
\end{align}
where the inequality follows from applying Taylor's theorem to the function $h(x)=(x - (1+x)\ln(1+x))$.
We also have
\begin{align}
\ln (g_{a}(t \pm \epsilon)) & = \mu((1+\delta')e^{\pm\epsilon} - 1) - (\mu(1+\delta'))(\ln(1+\delta') \pm \epsilon) \nonumber\\
& \le \mu((1+\delta')(1\pm \epsilon+\epsilon^2) - 1) - (\mu(1+\delta'))(\ln(1+\delta') \pm \epsilon) \nonumber\\
& = \mu \delta' - \mu(1+\delta')\ln(1+\delta') + \mu (1+\delta')\epsilon^2 \nonumber\\
& = \ln [g_{a}(t)] + 1 \label{eqn:gpmeps}
,\end{align}
where the inequality follows from Taylor's theorem and the fact that $e^{\pm \epsilon} \le e^{\frac{1}{3\sqrt{3}}} \le 2$ and the last equality uses the fact that $\mu (1+\delta')\epsilon^2 = \mu (1+\delta') (1/{\sqrt{\mu (1+\delta')}})^2 = 1$.
(Inequality (\ref{eqn:gpmeps}) is shorthand for two inequalties, one (resp.\ the other) with ${}+{}$ (resp.\ ${}-{}$) substituted for $\pm$.)

Putting the pieces together we get
\begin{align}
\pr{Y > (1+\delta)\mu} &= \pr{Y > a - u} \nonumber\\
& \ge e^{-\ln(1+\delta') 3\sqrt{3\mu}} \left[g_{a}(\ln(1+\delta')) - e^{-\epsilon 3\sqrt{3\mu}} (g_{a}(\ln(1+\delta') + \epsilon) + g_{a}(\ln(1+\delta') - \epsilon))   \right] \nonumber\\
& \ge e^{-\ln(1+\delta') 3\sqrt{3\mu}} g_{a}(\ln(1+\delta'))\left[1 - e^{-\epsilon 3\sqrt{3\mu}} 2e  \right] \nonumber\\
& \ge e^{-\delta' 3\sqrt{3\mu}} g_{a}(\ln(1+\delta')) \left[1 - e^{-3} 2e \right] \nonumber\\
& \ge e^{-\delta' 3 \sqrt{3\mu}} e^{-\mu \delta'^2/2} \left[1 - e^{-3} 2e \right] \nonumber\\
& \ge e^{-\delta' 3 \sqrt{3\mu}-\mu \delta'^2/2 - 1} \label{eqn:pieces}
\end{align}
where the first inequality uses (\ref{eqn:A21}), the second inequality uses (\ref{eqn:gpmeps}), the third inequality uses $\ln(1+\delta') \le \delta'$ and $\epsilon 3\sqrt{3\mu} = \frac{1}{\sqrt{\mu(1+\delta')}} \cdot 3 \sqrt{3\mu} \ge 3 \sqrt{3/3} = 3$, and the fourth inequality uses (\ref{eqn:g}). Finally we bound
\begin{align}
\delta' 3 \sqrt{3\mu}+\mu \delta'^2/2 + 1 &= (\delta  + 3\sqrt{3/\mu}) 3\sqrt{3\mu} + \mu(\delta + 3\sqrt{3/\mu})^2/2 + 1 \nonumber\\
&=3\sqrt{3}x + 27 + x^2/2 + 3\sqrt{3}x + 27/2 + 1 \nonumber\\
&\le x^2/2 + 6\sqrt{3}x + 42 \nonumber\\
& \le x^2 + (6\sqrt{3})^2/2 + 42 = \delta^2 \mu + 96 \label{eqn:96}
\end{align}
where $x=\delta \sqrt{\mu}$. Equations (\ref{eqn:pieces}) and (\ref{eqn:96}) imply that
\begin{align}
\pr{Y \ge \Eb{Y} + \lambda} &\ge \pr{Y > (1+\delta)\mu} \nonumber \\
&\ge e^{-96-\frac{\lambda^2}{\mu}} \label{eqn:binomLB2}
.\end{align}

To complete the proof of the lemma we use a limiting argument. Let $Z_1,Z_2,\ldots$ be random variables where $Z_n$ has binomial distribution $B(n, \mu/n)$. Straightforward calculation (e.g.\ Theorem 5.5 in \cite{mitzenmacher}) shows that $\lim_{n \to \infty} \pr{Z_n=i} = \pr{Y=i}$ for any integer $i$ (i.e.\ $Z_n$ converges in distribution to $Y$).  It follows that $\lim_{n \to \infty} \pr{Z_n \ge i} = 1 - \lim_{n \to \infty}\sum_{j=0}^{i-1} \pr{Z_n = j} = 1 - \sum_{j=0}^{i-1} \pr{Y = j} = \pr{Y \ge i}$. Consequently (choose $i=\ceil{\mu + \lambda}$) there exists $n' \ge 0$ such that $|\pr{Y \ge \mu + \lambda} - \pr{Z_{n'} \ge \mu + \lambda}| \le |e^{-96-\frac{\lambda^2}{\mu}}-e^{-100-\frac{\lambda^2}{\mu}}|$. Combining this fact with (\ref{eqn:binomLB2}) yields $\pr{Z_n \ge \mu + \lambda} \ge e^{-100-\frac{\lambda^2}{\mu}}$, i.e.\ $Z=Z_{n'}$ satisfies (\ref{eqn:binomLB}).

\end{document}